\documentclass[reqno,a4paper]{amsart}
\usepackage{eucal,amsfonts,amssymb,amsmath,amsthm,epsfig,mathrsfs,subeqnarray}

\usepackage{amscd,amsxtra}
\usepackage{enumerate}
\usepackage{latexsym}

\allowdisplaybreaks

 \makeatletter \@addtoreset{equation}{section}

\makeatother \makeatletter

\newtheorem{thm}{Theorem}[section]
{\rm}
{\rm}
\newtheorem{lem}[thm]{Lemma}
\newtheorem{coro}[thm]{Corollary}
\newtheorem{prop}[thm]{Proposition}
\newtheorem{defi}[thm]{Definition}
{\rm}

\newcommand{\C}{{\mathbb C}}

\newcommand{\bd}{\begin{defi}}
\newcommand{\ed}{\end{defi}}

\newcommand{\be}{\begin{equation}}
\newcommand{\ee}{\end{equation}}
\newcommand{\barr}{\begin{array}}
\newcommand{\earr}{\end{array}}
\newcommand{\bmn}{\begin{eqnarray}}
\newcommand{\emn}{\end{eqnarray}}
\newcommand{\bnm}{\begin{eqnarray*}}
\newcommand{\enm}{\end{eqnarray*}}
\newcommand{\bln}{\begin{subequations}}
\newcommand{\eln}{\end{subequations}}

\newcommand{\ba}{\begin{align}}
\newcommand{\ea}{\end{align}}
\newcommand{\banm}{\begin{align*}}
\newcommand{\eanm}{\end{align*}}

\newcommand{\e}{\varepsilon}
\newcommand{\ve}{\varepsilon}

\title[Asymptotic Analysis in a Gas-Solid Combustion Model]{Asymptotic Analysis in a Gas-Solid Combustion Model with Pattern Formation}
\author[C.-M. Brauner]{Claude-Michel Brauner}
\address{School of Mathematical Sciences, Xiamen University, 361005 Xiamen, China, and
Institut de Math\'ematiques de Bordeaux, Universit\'e de Bordeaux, 33405 Talence cedex, France}
\email{cmbrauner@gmail.com, claude-michel.brauner@u-bordeaux1.fr}
\author[L. Hu]{Lina Hu}
\address{School of Mathematical Sciences, Xiamen University, 361005 Xiamen, China}
\email{linahu@stu.xmu.edu.cn}
\author{L. Lorenzi}
\address{Dipartimento di Matematica e Informatica, Universit\`a degli Studi di Parma, Parco Area delle Scienze 53/A, I-43124 Parma, Italy.}
\email{luca.lorenzi@unipr.it}
\thanks{This work was partially supported by a grant of the Fujian Administration of Foreign Expert Affairs, China.}

\begin{document}

\begin{abstract}
We consider a free interface problem which stems from a solid-gas model in combustion
with pattern formation. We derive a third-order, fully nonlinear, self-consistent equation for the flame front.
Asymptotic methods reveal that the interface approaches a solution of the Kuramoto-Sivashinsky equation. Numerical results are presented which illustrate the dynamics.
\end{abstract}

\maketitle

\section{Introduction}\label{intro}
Combustion phenomena are particularly important for Science and Industry, as J.-L. Lions pointed it out in his foreword to the special issue of the CNRS {\it "Images des Math\'ematiques"} in 1996 (\cite{image}). Flames constitute a complex physical system involving fluid dynamics and multistep chemical kinetics (see e.g., \cite{BuckLud}). In the middle of the 20th century, the Russian School, which included Frank-Kamenetskii and Zel'dovich, used formal asymptotics based on large activation energy to write simpler descriptions of such a reactive system. Later, the development of systematic asymptotic techniques during the 1960s opened the way towards revealing an underlying simplicity in many combustion processes. Eventually, the full power of asymptotical analysis has been realized by modern singular perturbation theory. \hbox{J.-L. Lions} was the first one to put these formalities on a rigorous basis in his seminal monograph {\it ``Perturbations singuli\`eres dans les probl\`emes aux limites et en contr\^ole optimal''} (\cite{lions}).

In short, the small perturbation parameter in activation-energy asymptotics is the inverse of the normalized activation energy, the  Zel'dovich number $\beta$. In the limit $\beta \to +\infty$, the flame front reduces to a free interface.
The laminar flames of low-Lewis-number premixtures are known to display diffusive-thermal instability
responsible for the formation of a non-steady cellular structure (see \cite{S83}), when the Lewis number $Le$
(the ratio of thermal and molecular diffusivities) is such that $Le \lesssim 1$.
From an asymptotical viewpoint, one combines the limit of large activation energy with the requirement that $\alpha=\frac{1}{2}\beta(1-Le)$ remains bounded: in the Near Equidiffusive Flame model (NEF), $\beta^{-1}$  and $1-Le$ are asymptotically of the same order of magnitude (see \cite{mat-siva}).

A very challenging problem is the derivation of a single equation for the (perturbation of the) free interface, which may capture most of the dynamics and, as a consequence, yields a reduction of the effective dimensionality of the system. Asymptotical methods are again the main tool: in a set of conveniently rescaled dependent and independent variables, the flame front is  asymptotically represented (see \cite{siva}) by a solution of the Kuramoto-Sivashinsky equation:
\begin{equation}
\label{eqn:KS} \Phi_{\tau} + 4\Phi_{\eta\eta\eta\eta} +
\Phi_{\eta\eta} + \frac{1}{2}({\Phi_{\eta})}^2 = 0.
\tag{K-S}
\end{equation}
This equation has received considerable attention from the mathematical
community (see \cite{temam}), especially for its ability to generate a cellular structure, pattern formation,
and chaotic behavior in appropriate range of parameters (see \cite{HN}).
We refer to \cite{BFHLS}-\cite{BLSX10} for a rigorous mathematical approach to the derivation of (K-S).

In this paper, we consider a model in gas-solid combustion, proposed in \cite{KagSiv08}. This model was motivated by experimental studies of Zik and Moses (see \cite{ZM}), who observed a striking fingering pattern in flame spread over thin solid fuels. The phenomenon was interpreted in terms of the diffusive instability similar to that occurring in laminar flames of low-Lewis-number premixtures.
As we show below, the gas-solid and premixed gas systems share
some common asymptotic features, especially K-S equation.

The free interface system for the scaled temperature $\theta$, the excess enthalpy $S$, the prescribed flow intensity $U$ (with $0<U<1$), and the moving front $x=\xi(t,y)$, reads:
\begin{align}
&U\frac{\partial\theta}{\partial x}
= \Delta \theta, \qquad\;\;x < \xi (t,y),
\label{eqn:FBP1}
\\[1mm]
&\theta = 1,\qquad\;\; x \geq \xi (t,y),
\label{eqn:FBP2}
\\[1mm]
&\frac{\partial \theta}{\partial t} + U\frac{\partial S}{\partial x}= \Delta S  -
\alpha\Delta\theta,\qquad \;\; x \neq \xi (t,y).
\label{eqn:FBP3}
\end{align}
System \eqref{eqn:FBP1}-\eqref{eqn:FBP3} is coupled with the following jump conditions for the normal derivatives of $\theta$ and $S$:
\begin{equation}
\label{eqn:FBP4}
\bigg[\frac{\partial\theta}{\partial n}\bigg] = - \exp (S),\qquad\;\,
\bigg[\frac{\partial S}{\partial n}\bigg] =
\alpha \bigg[\frac{\partial\theta}{\partial n}\bigg].
\end{equation}

It is not difficult to show that System (\ref{eqn:FBP1})-(\ref{eqn:FBP4}) admits a planar
Traveling Wave solution, with velocity $-V=U\ln U$. Setting $x'=x+Vt$, the Traveling Wave solution reads:
\begin{align*}
&\overline{\theta}(x') = \left\{
\begin{array}{ll}
\exp (U x'), &x'\le 0,\\[1.5mm]
1, & x'> 0,
\end{array}
\right.\\[4mm]
&\overline{S}(x')= \left\{
\begin{array}{ll}
(\alpha -\ln U )U x' \exp (U x')+(\ln U)\exp(U x'), & x' \leq 0,\\[1.5mm]
\ln U, & x'> 0.
\end{array}
\right.
\end{align*}
As usual one fixes the moving front: we set
\begin{eqnarray*}
\xi(t,y)=-Vt +\varphi(t,y),\qquad\;\, x'=x-\xi(t,y),
\end{eqnarray*}
where $\varphi$ is the perturbation of the front.
In this new framework the system \eqref{eqn:FBP1}-\eqref{eqn:FBP3} reads:
\begin{align}
\label{temp-1}
& U \theta_{x'} = \Delta_\varphi \theta,\qquad\;\, x'<0,\\[1mm]
\label{temp-2}
&\theta =1,\qquad\;\, x'\ge 0,\\[1mm]
\label{enth} & \theta_t +(V-\varphi_t) \theta_{x'} + U S_{x'}= \Delta_\varphi S - \alpha
\Delta_\varphi \theta,\qquad\;\,x'\neq 0,
\end{align}
where
\begin{eqnarray*}
\Delta_\varphi = (1 + (\varphi_y)^2) D_{x'x'} + D_{yy}
-\varphi_{yy}D_{x'} -2\varphi_y D_{x'y}.
\end{eqnarray*}
The front is now fixed at $x'=0$.  The first jump condition
in (\ref{eqn:FBP4}) reads:
\begin{equation}
\sqrt{1+(\varphi_y)^2}\,\,\bigg[\frac{\partial\theta}{\partial
x'}\bigg] = - \exp (S),
\label{L1}
\end{equation}
the second one:
\begin{equation}
\bigg[\frac{\partial S}{\partial x'}\bigg] = \alpha
\bigg[\frac{\partial\theta}{\partial x'}\bigg].
\label{L2}
\end{equation}

We will consider a quasi-steady version of the model, motivated by the fact that, in similar problems,
not far from the instability threshold, the respective time
derivatives (if any) of the temperature and enthalpy exhibit a
relatively small effect on the solution: the dynamics appears to be
essentially driven by the front. We can thus introduce a \textit{quasi-steady} model replacing
(\ref{temp-1})-(\ref{enth}) by
\begin{align*}
& U \theta_{x'} = \Delta_\varphi \theta,\qquad\;\, x'<0,\\[1mm]
&\theta =1,\qquad x'\ge 0,\\[1mm]
&(V-\varphi_t) \theta_{x'} + U S_{x'}= \Delta_\varphi S - \alpha \Delta_\varphi
\theta,\qquad\, x'\neq 0.
\end{align*}
Next we consider the perturbations of temperature $u$ and enthalpy
$v$:
\begin{eqnarray*}
\theta = \overline{\theta}+u, \qquad\;\, S = \overline{S}+ v,
\end{eqnarray*}
and, for simplicity, in the equations satisfied by $u$ and $v$ and $\varphi$,
we keep only linear and second-order terms for $\varphi$, and first-order terms for  $u$ and $v$.
Writing $x$ instead of $x'$, to avoid cumbersome notation, some (easy) computations
reveal that the triplet $(u,v,\varphi)$ solves the differential equations
\begin{align*}
&Uu_x -\Delta u   = (\Delta_\varphi -
\Delta)\overline{\theta}, \qquad\;\, x<0,\\[1mm]
&Vu_x -\Delta(v-\alpha u) +Uv_{x}- \varphi_t \overline{\theta}_x =
(\Delta_\varphi - \Delta)(\overline{S} -\alpha \overline{\theta}),
\qquad\;\, x\neq 0,
\end{align*}
where $u\equiv 0$ in  $[0,+\infty)$ and
\begin{align*}
&(\Delta_\varphi - \Delta)\overline{\theta} = (U (\varphi_y)^2 -
\varphi_{yy})Ue^{Ux},\\[2mm]
&(\Delta_\varphi - \Delta)(\overline{S} -\alpha \overline{\theta})\\
=&
\left\{
\begin{array}{ll}
(\varphi_y)^2(\alpha-\ln U)U^2(1+Ux)e^{Ux}
 -\varphi_{yy}(\alpha-\ln U)U^2 x e^{Ux}, & x<0,\\[1.5mm]
 0, & x>0.
 \end{array}
 \right.
\end{align*}
The previous system is endowed with a set of boundary conditions.
First, the continuity of $\theta$ and $S$ at the front yields the equation
\begin{eqnarray*}
u(0^-)=[v]=0,
\end{eqnarray*}
(recall that $u(x)=0$ for $x\ge 0$). Second, condition \eqref{L1} gives, up to the second-order,
\begin{eqnarray*}
-U + [u_x]=-(1 +(\varphi_y)^2)^{-\frac{1}{2}}Ue^{v(0)} \sim -\left(1-
\frac{1}{2}(\varphi_y)^2\right) U \left(1+v(0) + \frac{1}{2}
(v(0))^2\right),
\end{eqnarray*}
and keeping only the first-order for $v$ we get the condition
\begin{align*}
&-u_x(0^-) +U v(0)= \frac{1}{2}(\varphi_y)^2 U.
\end{align*}
Finally, the condition $[S_x]=\alpha[\theta_x]$ yields
\begin{align*}
[v_x]=-\alpha u_x(0^-).
\end{align*}
Summing up, the final system is the following one:
\begin{equation}
\left\{
\begin{array}{ll}
Uu_x -\Delta u   = (\Delta_\varphi - \Delta)\overline{\theta}, & x<0,\\[1.5mm]
Vu_x -\Delta (v-\alpha u) +Uv_{x}- \varphi_t \overline{\theta}_x =  (\Delta_\varphi - \Delta)(\overline{S} -\alpha \overline{\theta}),\quad & x \neq 0,\\[1.5mm]
u(0^-)=[v]=0,\\[1.5mm]
Uv(0)-u_x(0^-) = \frac{1}{2}(\varphi_y)^2U,\\[1.5mm]
{\rm [}v_x{\rm ]}=-\alpha u_x(0^-).
\end{array}
\right. \label{final}
\end{equation}
Throughout this paper we will also use the very convenient notation:
\begin{eqnarray*}
\gamma = \alpha - \ln U.
\end{eqnarray*}

First, our goal is to derive a self-consistent equation for the front $\varphi$:
\begin{equation}\label{intro-3order}
\varphi_t={\mathscr A}(\varphi) +{\mathscr M}((\varphi_y)^2),
\end{equation}
where ${\mathscr A}$ is a third-order, pseudo-differential operator, in contrast to the NEF model in gaseous
combustion, where the corresponding linear operator is of the second-order (see \cite{BHLS10}).
Another important feature is that the nonlinear term is also of the third-order, which means that the equation
\eqref{intro-3order} is {\it fully nonlinear}. Here the spatial domain is a two-dimensional strip
$\mathbb{R} \times [-\ell/2, \ell/2]$, with periodic boundary conditions at $\pm \ell/2$.

Second, we define a small parameter $\e=\gamma -1$. The main result of this paper states in which precise
sense the front $\varphi$ approaches a solution of the Kuramoto-Sivashinsky equation when
$\e \to 0$:

{\bf Main Theorem}
{\em
Let $\Phi_0 \in H^m(-\ell_0/2,\ell_0/2)$ be a periodic function of period
$\ell_0$. Further, let $\Phi$ be the periodic solution of
\eqref{eqn:KS} $($with period $\ell_0)$ on a fixed time interval
$[0,T]$, satisfying the initial condition $\Phi(0,\cdot)=\Phi_0$.
If $m$ is large enough, then there exists $\e_0=\e_0(T)\in (0,1)$
such that, for $0<\e\leq\e_{0}$, Equation \eqref{intro-3order} admits
a unique classical solution $\varphi$ on $[0,\frac{T}{\varepsilon^2U^2}]$,
which is periodic with period $\frac{\ell_{0}}{\sqrt{\varepsilon}U}$ with
respect to $y$, and satisfies
\begin{eqnarray*}
\varphi(0,y)=\e U^{-1}\Phi_0(y\sqrt{\varepsilon}U),\qquad\;\,|y|\le
\frac{\ell_0}{2\sqrt{\e}U}.
\end{eqnarray*}
Moreover, there exists a positive constant $C$ such that
\begin{eqnarray*}
|\varphi(t,y)-\varepsilon U^{-1}
\Phi(t\varepsilon^{2}U^2,y\sqrt{\varepsilon}U)|\leq
C\,\varepsilon^{2} ,\qquad\;\, 0\leq
t\leq\frac{T}{\varepsilon^{2}U^2},\;\,|y|\leq\frac{\ell_{0}}{2\sqrt{\varepsilon}U},
\end{eqnarray*}
for any $\varepsilon\in (0,\e_0]$.}

\medskip

The paper is organized as follows:  in Section \ref{ansatz}, we proceed to a formal Ansatz in the spirit
of  \cite{siva}, defining the rescaled variable $\psi=U\varphi/\e$ and expanding $\psi=\psi^0+\e\psi^1+\ldots$.
It transpires that $\psi^0$ verifies \eqref{eqn:KS}, thanks to an elementary solvability
condition.

Section \ref{third-order} is devoted to the derivation of \eqref{intro-3order},
via an explicit computation in discrete Fourier variable. The asymptotic analysis in the rescaled
variables $t= \tau/\varepsilon^2 U^2, y = \eta/\sqrt{\varepsilon} U$ is performed in Section \ref{asymptotical}.
Since the perturbation in \eqref{intro-3order} is singular as $\e \to 0$,
we turn to the equivalent (at fixed $\e>0$) fourth-order, fully nonlinear equation \eqref{perturbed-4th-order-intro},
whose prima facie limit as $\e \to 0$ is Equation \eqref{eqn:KS}:
\begin{align}\label{perturbed-4th-order-intro}
\frac{\partial}{\partial\tau} &\left (\sqrt{I-4\e D_{\eta\eta}} \right)\psi
=-4D_{\eta\eta\eta\eta}\psi-D_{\eta\eta}\psi \nonumber\\[1mm]
&+\frac{1}{4}\left\{(I-4\e D_{\eta\eta})^{\frac{3}{2}}-3(I-4\e
D_{\eta\eta})-4(1+\e)\left (\sqrt{I-4\e D_{\eta\eta}}-I\right )
\right\}(D_{\eta}\psi)^2.
\end{align}
We prove a priori estimates which constitute the key tool to prove the Main Theorem.
Finally, numerical computations which illustrate the dynamics in Equation \eqref{perturbed-4th-order-intro}
are presented in Section \ref{numerics}.

The stability issue in System \eqref{eqn:FBP1}-\eqref{eqn:FBP4} will be addressed in a forthcoming paper, using the methods of \cite{BL} and \cite{lorenzi-1}-\cite{lorenzi-lunardi}.

\paragraph*{Notation}
Given a (smooth enough) function $f:(-\ell/2,\ell/2)\to\mathbb C$,
we denote by $\widehat f(k)$ its $k$-th Fourier coefficient, that
is, we write
\begin{eqnarray*}
f(y)=\sum_{k=0}^{+\infty}\widehat f(k) w_k(y),\qquad\;\,y\in (-\ell/2,\ell/2),
\end{eqnarray*}
where $\{w_{k}\}$ is a complete set of (complex valued)
eigenfunctions of the operator
\begin{eqnarray*}
D_{yy}:{H}^{2}(-\ell/2,\ell/2)\,\to\, {L}^{2}(-\ell/2,\ell/2),
\end{eqnarray*}
whose eigenvalues $0,-\frac{4\pi^2}{\ell^2},-\frac{4\pi^2}{\ell^2},-\frac{16\pi^2}{\ell^2},-\frac{16\pi^2}{\ell^2},-\frac{36\pi^2}{\ell^2},\dots$
we label as
$0=-\lambda_0(\ell)>-\lambda_1(\ell)=-\lambda_2(\ell)>-\lambda_3(\ell)=-\lambda_4(\ell)>\dots$
Typically, when no confusion may arise, we simply write $\lambda_k$ instead of $\lambda_k(\ell)$.

For any $s\ge 0$ we denote by $H^{s}_{\sharp}$ the usual Sobolev
space of order $s$ consisting of $\ell$-periodic (generalized) functions, i.e.,
\begin{eqnarray*}
H_{\sharp}^{s}=\left\{u=\sum_{k=0}^{+\infty}\widehat u(k)
w_{k}:\,\sum_{k=0}^{+\infty}\lambda_{k}^{s}
|\widehat u(k)|^{2}<+\infty\right\}.
\end{eqnarray*}
For $s=0$, we simply write $L^2$ instead of $H^0_{\sharp}$ and we
denote by $|\cdot|_2$ the usual $L^2$-norm.

By the notation $\widehat f(x,k)$ we mean the $k$-th Fourier coefficient of the function $f(x,\cdot)$. A similar notation is
used for functions which depend also on the time variable.

\section{A formal Ansatz}\label{ansatz}
The aim of this section is to use a formal asymptotic expansion method,
in the spirit of \cite{siva}. The small perturbation parameter $\e>0$ is
defined by:
\begin{equation}\label{perturbation}
\alpha = 1 + \ln U + \e, \quad {\rm i.e.,} \quad \gamma = 1+ \e.
\end{equation}
Accordingly, we now introduce scaled dependent and independent variables:
\begin{equation}\label{rescaled var}
t= \frac{\tau}{\varepsilon^2 U^2},\qquad\;\, y = \frac{\eta}{\sqrt{\varepsilon} U}, \qquad\;\,
\varphi = \frac{\varepsilon}{U} \psi, \qquad\;\, u=\varepsilon^2 u_1, \qquad\;\, v=\varepsilon^2 v_1,
\end{equation}
and the Ansatz:
\begin{eqnarray*}
u_1=u_1^0+\varepsilon u_1^1+\ldots, \qquad\;\, v_1=v_1^0+\varepsilon v_1^1+\ldots,
\qquad\;\,\psi=\psi^0+\varepsilon \psi^1+\ldots.
\end{eqnarray*}
It is easy to rewrite System \eqref{final} in terms of the rescaled variables.
At the zeroth order, it comes:
\begin{equation}
\left\{
\begin{array}{ll}
U(u_1^0)_x - (u_1^0)_{xx} =-U^2e^{Ux}\psi_{\eta\eta}^0, & x<0,\\[2mm]
V(u_1^0)_x -(v_{1}^0)_{xx} +(u_{1}^0)_{xx}+(\ln U)(u_{1}^0)_{xx}+U(v_{1}^0)_x= -U^3x e^{Ux}\psi_{\eta\eta}^0,\quad & x< 0,\\[2mm]
u_1^0=0, & x\ge 0,\\[2mm]
(v_{1}^0)_{xx}-U(v_{1}^0)_x=0, & x> 0.
\end{array}
\right.\label{final-0}
\end{equation}
At $x=0$, the following conditions should be satisfied:
\begin{subeqnarray}
\slabel{intefv01}
&u_1^0(0)=\big[ v_1^0 \big]=0,\\[1mm]
\slabel{interfv02}
&(u_{1}^0)_x(0)-Uv_1^0(0)=0,\\[1mm]
\slabel{intefv03}
&\big[ (v_{1}^0)_x \big]=-(1+\ln U)(u_{1}^0)_x(0).
\end{subeqnarray}
We assume that the functions $x\mapsto e^{-Ux/2}u_1^0(x)$ and $x\mapsto e^{-Ux/2}v_1^0(x)$ are bounded in $(-\infty,0)$ and
in $\mathbb R$, respectively.
Note that System \eqref{final-0} coupled with conditions \eqref{intefv01} and \eqref{interfv02} is
uniquely solvable in the unknowns $(u_1^0,v_1^0)$, taking $\psi^0$ as a parameter.
It turns out that
\begin{align*}
&u_1^0=Uxe^{Ux}\psi_{\eta\eta}^0,\qquad\;\, x < 0,\\[1mm]
&v_1^0=e^{Ux}\psi_{\eta\eta}^0 +U(\ln U)xe^{Ux}\psi_{\eta\eta}^0+U^2x^2e^{Ux}\psi_{\eta\eta}^0 ,\qquad\;\, x < 0,\\[1mm]
&u_1^0=0,\qquad\;\, x\ge 0,\\[1mm]
&v_1^0=\psi_{\eta\eta}^0,\qquad\;\, x\ge 0.
\end{align*}
One might be tempted to use condition \eqref{intefv03} to determine function $\psi^0$.
Unfortunately, whichever $\psi^0$ is, the triplet $(u^0_1,v^0_1,\psi^0)$ satisfies
this condition. As a matter of fact, we are not able to determine uniquely a solution
to Problem \eqref{final-0}-\eqref{intefv03}.
This situation is not surprising at all in singular perturbation theory, see \cite{eckhaus}, \cite{lions}.
To determine $\psi^0$ one needs
to consider the (linear) problem for the first-order terms in the asymptotic expansion of $u_1$, $v_1$ and $\psi$.
As we show in a while, this problem provides a solvability condition, which is just the missing equation for $\psi^0$.

The system for $(u_1^1,v_1^1,\psi^1)$ is the following one:
\begin{equation}
\left\{
\begin{array}{l}
U(u_{1}^1)_x - (u_{1}^1)_{xx} -U^2(u_{1}^0)_{\eta\eta}=(U(\psi_{\eta}^0)^2-U\psi_{\eta\eta}^1)Ue^{Ux},\qquad\qquad x<0,\\[3mm]
V(u_{1}^1)_x -(v_{1}^1)_{xx} -U^2(v_{1}^0)_{\eta\eta}+(u_{1}^0)_{xx}+(1+\ln U)\big ((u_{1}^1)_{xx}+U^2(u_{1}^0)_{\eta\eta}\big )+U(v_{1}^1)_x\\[2mm]
=U^2\psi_{\tau}^0e^{Ux}+ (\psi_{\eta}^0)^2U^2 e^{Ux}+ (\psi_{\eta}^0)^2U^3xe^{Ux}-U^3\psi_{\eta\eta}^1xe^{Ux}-U^3\psi_{\eta\eta}^0xe^{Ux},\\[2mm]
\qquad\qquad\qquad\qquad\qquad\qquad\qquad\qquad\qquad\qquad\qquad\quad\qquad\qquad\quad x<0,\\[3mm]
(v_{1}^1)_{xx}+U^2(v_{1}^0)_{\eta\eta}-U(v_{1}^1)_x=0,\qquad\qquad\qquad\qquad\qquad\qquad\quad\;\,\,  x> 0,\\[2mm]
u_1^1=0,  \qquad\qquad\qquad\qquad\qquad\qquad\qquad\qquad\qquad\;\;\,\qquad\qquad\qquad x\ge 0,\\[2mm]
u_1^1(0)=\big[ v_1^1 \big]=0,\\[2mm]
Uv_1^1(0)-(u_{1}^1)_x(0)=\frac{1}{2}U(\psi_{\eta}^0)^2,\\[2mm]
\big[ (v_{1}^1)_x \big]=-(1+\ln U) (u_{1}^1)_x(0)- (u_{1}^0)_x(0).
\end{array}
\right.\label{final-u1111}
\end{equation}
As above, we assume that the functions $x\mapsto e^{-Ux/2}u_1^1(x)$ and $x\mapsto e^{-Ux/2}v_1^1(x)$ are bounded
in $(-\infty,0)$ and in $\mathbb R$, respectively.
Using these conditions one can easily show that the more general solutions $(u_1^1,v_1^1,\psi^1)$
to the differential equations and the first boundary condition in \eqref{final-u1111} are given by
\begin{align*}
&u_1^1=Uxe^{Ux}(\psi_{\eta\eta\eta\eta}^0-(\psi_{\eta}^0)^2+\psi_{\eta\eta}^1 )-\frac{1}{2}U^2x^2e^{Ux}\psi_{\eta\eta\eta\eta}^0,\qquad\;\,x < 0,\\[2mm]
&v_1^1=v_1^1(0)e^{Ux}+Axe^{Ux}+Bx^2e^{Ux}+Cx^3e^{Ux},\qquad\;\, x < 0,\\[2mm]
&u_1^1=0,\qquad\;\, x \ge 0,\\[2mm]
&v_1^1=v_1^1(0)+Ux\psi_{\eta\eta\eta\eta}^0,\qquad\;\, x\ge 0,
\end{align*}
where
\begin{align*}
A=&U(\ln U)(\psi_{\eta\eta}^1-(\psi_{\eta}^0)^2+\psi_{\eta\eta\eta\eta}^0)-U\psi_{\tau}^0-U(\psi_{\eta}^0)^2-3U\psi_{\eta\eta\eta\eta}^0,\\[2mm]
B=&U^2\psi_{\eta\eta}^0+U^2\psi_{\eta\eta}^1-U^2(\psi_{\eta}^0)^2-\frac{1}{2}U^2(\ln U)\psi_{\eta\eta\eta\eta}^0+\frac{3}{2}U^2\psi_{\eta\eta\eta\eta}^0,\\[2mm]
C=&-\frac{1}{2}U^3\psi_{\eta\eta\eta\eta}^0,
\end{align*}
and $v_1^1(0)$ is an arbitrary parameter. Hence, $(u^1_1,v^1_1)$ depends
on $\psi^1$. To determine both $\psi^1$ and $v_1^1(0)$ we use the last two boundary conditions, which give, respectively,
\begin{align}
\label{v20}
-U\psi_{\eta\eta\eta\eta}^0+U(\psi_{\eta}^0)^2-U\psi_{\eta\eta}^1+Uv_1^1(0)=\frac{U}{2}(\psi_{\eta}^0)^2
\end{align}
and
\begin{align}
U\psi_{\eta\eta}^1-Uv_1^1(0)=&-U\psi_{\tau}^0-U\psi_{\eta\eta}^0-5U\psi_{\eta\eta\eta\eta}^0,
\label{v20jump}
\end{align}
Obviously \eqref{v20} and \eqref{v20jump} is a linear system for $(v_1^1(0),\psi_{\eta\eta}^1)$ with solvability condition:
\begin{align*}
\psi_{\tau}^0+\psi_{\eta\eta}^0+4\psi_{\eta\eta\eta\eta}^0+\frac{1}{2}(\psi_{\eta}^0)^2=0,
\end{align*}
i.e., $\psi^0$ verifies a K-S equation. Hence, the Kuramoto-Sivashinsky equation is the missing equation, at the zeroth-order,
needed to uniquely determine $(u_1^0,v_1^0,\psi^0)$.

\section{A third-order fully nonlinear pseudo-differential equation for the front}\label{third-order}
\setcounter{equation}{0}
The aim of this section is the derivation of a self-consistent pseudo-differential equation for the front $\varphi$.  We consider
System \eqref{final}, namely:
\begin{equation}
\left\{
\begin{array}{ll}
Uu_x -\Delta u =Ue^x( U(\varphi_y)^2 -\varphi_{yy}), & x<0,\\[2mm]
Vu_x -\Delta (v-\alpha u)+Uv_{x} = U\varphi_te^x +(\alpha-\ln U)U^2(1+Ux)e^{Ux}(\varphi_y)^2\\[1.5mm]
\qquad\qquad\qquad\qquad\qquad\qquad -U^2(\alpha-\ln U)x e^{Ux}\varphi_{yy} ,\qquad & x< 0,\\[2mm]
Uv_x -\Delta v = 0, & x > 0,\\[2mm]
u(0)=[v]=0,\\[2mm]
Uv(0)-u_x(0) = \frac{1}{2}U(\varphi_y)^2,\\[2mm]
{\rm [}v_x{\rm ]}=-\alpha u_x(0),
\end{array}
\right.
\label{final-u-1}
\end{equation}
in a two-dimensional strip $\mathbb{R} \times [-\ell/2, \ell/2]$, with periodicity in the $y$ variable.

\subsection{Computations in discrete Fourier variable}
Throughout this section, $(u,v,\varphi)$ is a sufficiently
smooth solution of System \eqref{final-u-1} such that the functions
\begin{eqnarray*}
(x,y)\mapsto e^{-Ux/2}u(t,x,y),\qquad\;\,(x,y)\mapsto e^{-Ux/2}v(t,x,y)
\end{eqnarray*}
are bounded in $(-\infty,0]\times [-\ell/2,\ell/2]$ and in $\mathbb R\times [-\ell/2,\ell/2]$, respectively.

We start from the first equation in \eqref{final-u-1}, namely:
\begin{equation}
Uu_x -\Delta u =( U(\varphi_y)^2 -\varphi_{yy})Ue^x,
\label{LL}
\end{equation}
and the boundary condition $u(\cdot,0,\cdot)=0$.
Applying the Fourier transform to both the sides of  \eqref{LL} we end up with the infinitely many
equations
\begin{eqnarray*}
U\widehat u_x(t,x,k)-\widehat u_{xx}(t,x,k)+\lambda_k\widehat u(t,x,k)=\left
( U\widehat{(\varphi_y)^2}(t,k)
+\lambda_k\widehat\varphi(t,k)\right )Ue^{Ux},
\end{eqnarray*}
for $k\ge 0$. For notational
convenience we set $\nu_k=\frac{U}{2}+\frac{1}{2}\sqrt{U^2+4\lambda_k}$ for any $k\ge 0$.

Since $u$ vanishes at $x=0$ and tends to $0$ as $x\to -\infty$ not slower than $e^{Ux/2}$,
the modes $\widehat u(\cdot,\cdot,k)$ should enjoy the same properties. Easy computations reveal that
\begin{align*}
&\widehat u(t,x,0)=-U\widehat{(\varphi_y)^2}(t,0)xe^{Ux},\qquad\;\,x\le 0,\\[1mm]
&\widehat u(t,x,k)=U(\lambda_k)^{-1} \left ( U\widehat{(\varphi_y)^2}(t,k)+\lambda_k\widehat{\varphi}(t,k)\right )\left (e^{Ux}-e^{\nu_kx}\right ),
\quad\;\,x\le 0,\;\, k\ge 1.
\end{align*}

Applying the same arguments to the equation for v jointly with the second- and the fourth-
boundary conditions in \eqref{final-u-1}, we obtain that the modes $\widehat v(\cdot,\cdot,k)$ are given by
\begin{align*}
\widehat v(t,x,0)= &\frac{1}{U}\widehat\varphi_t(t,0)+(-\gamma U \widehat{(\varphi_y)^2}(t,0)-U(\ln U) \widehat{(\varphi_y)^2}(t,0)-\widehat{\varphi_t}(t,0))
x e^{Ux}\\
&-\gamma U^{2}\widehat{(\varphi_y)^2}(t,0)x^2e^{Ux},\qquad\;\,x<0,\\[3mm]
\widehat v(t,x,0)=&\frac{1}{U}\widehat\varphi_t(t,0),\qquad\;\,x>0
\end{align*}
and
\begin{align*}
\widehat v(t,x,k)=&c_{1,k}e^{\nu_kx}+ A_ke^{Ux}+B_kxe^{Ux}+C_kxe^{\nu_k x},\qquad\;\,x<0,\\[3mm]
\widehat v(t,x,k)=&c_{2,k}e^{(U-\nu_k)x},\qquad\;\,x\ge 0,
\end{align*}
for $k\ge 1$, where
\begin{align*}
A_k=&\frac{(\alpha+\gamma) U^2}{\lambda_k}\widehat{(\varphi_y)^2}(t,k)+\alpha U\widehat\varphi(t,k)+\frac{U}{\lambda_k}\widehat\varphi_t(t,k),\qquad\;\\[4mm]
B_k=&\frac{\gamma U^3}{\lambda_k}\widehat{(\varphi_y)^2}(t,k)+\gamma U^2\widehat\varphi(t,k),\qquad\;\\[4mm]
C_k=&\frac{\gamma U^3\nu_k}{\lambda_k(U-2\nu_k)}\widehat{(\varphi_y)^2}(t,k)+\frac{\gamma U^2\nu_k}{U-2\nu_k}\widehat\varphi(t,k),\\[4mm]
c_{2,k}=&\left(\frac{\gamma U^2(U-\nu_k)}{\lambda_k(U-2\nu_k)}+\frac{\gamma U^3}{\lambda_k(U-2\nu_k)}+\frac{\gamma U^3}{(\nu_k-U)(U-2\nu_k)^2} \right)\widehat{(\varphi_y)^2}(t,k)\\[2mm]
&+\left( \frac{\gamma U^2}{U-2\nu_k}+\frac{\gamma U^2\nu_k}{(U-2\nu_k)^2}\right)\widehat\varphi(t,k)
+ \frac{U(U-\nu_k)}{\lambda_k(U-2\nu_k)}\widehat\varphi_t(t,k),\\[4mm]
c_{1,k}=& c_{2,k}-A_k.
\end{align*}

The equation for the front now comes writing the last but one boundary condition in \eqref{final-u-1}, which we have no used so far,
in Fourier variable, and taking advance of the formulas for the modes of $\widehat u$ and $\widehat v$.
It turns out that the equation for the front (in Fourier coordinates) reads:
\begin{align*}
&\widehat\varphi_t(t,0)+\frac{1}{2}U\widehat{(\varphi_y)^2}(t,0)=0,\\[3mm]
& \frac{U(U-\nu_k)}{\lambda_k(U-2\nu_k)}\widehat\varphi_t(t,k)
+\left( \frac{\gamma U^2}{U-2\nu_k}+\frac{\gamma U^2\nu_k}{(U-2\nu_k)^2}+\nu_k-U\right)\widehat\varphi(t,k)\\[1mm]
&+ \left(\frac{\gamma U^2(U-\nu_k)}{\lambda_k(U-2\nu_k)}+\frac{\gamma U^3}{\lambda_k(U-2\nu_k)}+\frac{\gamma U^3}{(\nu_k-U)(U-2\nu_k)^2} +\frac{U(\nu_k-U)}{\lambda_k}-\frac{1}{2}\right)\\[1mm]
&\qquad\quad\times\widehat{(\varphi_y)^2}(t,k)=0,
\end{align*}
or, even, in the much more compact form
\begin{align}
(X_k U)\widehat\varphi_t(t,k)=&\frac{1}{4}(U^2-X_k^2)(X_k^2-\gamma U^2)\widehat\varphi(t,k)\\[1.5mm]
&+\frac{1}{4}(X^3_k-3UX_k^2-4\gamma U^2 X_k +4\gamma U^3)\widehat{(\varphi_y)^2}(t,k)\notag\\[1.5mm]
=&(-4\lambda_k^2+(\gamma-1)U^2\lambda_k)\widehat\varphi(t,k)\notag\\[1.5mm]
&+\frac{1}{4}(X^3_k-3UX_k^2-4\gamma U^2X_k+4\gamma U^3)\widehat{(\varphi_y)^2}(t,k),
\label{last-feb}
\end{align}
for any $k\ge 0$, if we set
\begin{eqnarray*}
X_k=\sqrt{U^2+4\lambda_k},\qquad\;\,k\ge 0.
\end{eqnarray*}

Therefore, we have proved the following

\begin{prop}\label{fourier-eqn}
Let $(u,v,\varphi)$ be a sufficiently
smooth solution of System \eqref{final-u-1} such that the functions $(x,y)\mapsto e^{-Ux/2}u(t,x,y)$ and $(x,y)\mapsto e^{-Ux/2}v(t,x,y)$
are bounded in $(-\infty,0]\times [-\ell/2,\ell/2]$ and in $\mathbb R\times [-\ell/2,\ell/2]$, respectively.
Then, the interface $\varphi$ solves the equations \eqref{last-feb}
for any $k\ge 0$.
\end{prop}
\subsection{A fourth-order pseudo-differential equation for the front}
Let us define the pseudo-differential operators (or Fourier multipliers)
${\mathscr B}, {\mathscr L}$ and ${\mathscr
F}$ through their symbols, respectively
\begin{align*}
b_k&=X_k U,\\
l_k&=-4\lambda_k^2+(\gamma-1)\lambda_k U^2,\\
f_k&=\frac{1}{4}(X^3_k-3UX_k^2-4\gamma U^2X_k +4\gamma U^3),
\end{align*}
for any $k\ge 0$.
It is easy to see that
\begin{align*}
{\mathscr B}&=U (U^2 I-4D_{yy})^{\frac{1}{2}} ,\\
{\mathscr
F}&=\frac{1}{4}(U^2 I-4D_{yy})^{\frac{3}{2}}-\frac{3}{4}U(U^2 I-4D_{yy})-\gamma U^2\left
(\sqrt{U^2 I-4D_{yy}}-U\right ),
\end{align*}
while the realization of ${\mathscr L}$ in $L^2$ is the
operator
\begin{align*}
L=-4D_{yyyy}-(\gamma-1)U^2D_{yy},
\end{align*}
with $H^4_{\sharp}$ as a domain.

It follows from Proposition \ref{fourier-eqn} that the front
$\varphi$ solves the equation
\begin{equation}
\frac{d}{dt}{\mathscr B}(\varphi)={\mathscr L}(\varphi)+{\mathscr
F}((\varphi_y)^2).
\label{eq-4-order}
\end{equation}
The main feature of Equation \eqref{eq-4-order} is that the
nonlinear part is rather unusual. Actually, it has a fourth-order
leading term, as ${\mathscr L}$ has. Therefore,  \eqref{eq-4-order}
is  a {\it fully nonlinear equation}. More precisely, we have:
\begin{lem}\label{B-invert}
The operators ${\mathscr B}$ and ${\mathscr F}$ admit
bounded realization $B:H^1_{\sharp}\to L^2$
and $F:H^3_{\sharp}\to L^2$, respectively.
Moreover, $B$ is invertible.
\end{lem}
\begin{proof}
A straightforward asymptotic analysis reveals that
\begin{eqnarray*}
b_k\sim 2\sqrt{\lambda_k}U,\qquad\;\,f_k\sim 2\lambda_k^{3/2},
\end{eqnarray*}
as $k\to +\infty$, from which we deduce that
${\mathscr B}$ and ${\mathscr F}$ admit bounded realizations $B:H^1_{\sharp}\to L^2$
and $F:H^3_{\sharp}\to L^2$.

Finally, since $b_k\neq 0$ for any $k\ge 0$, it follows that $B$ is invertible.
\end{proof}

\subsection{The third-order pseudo-differential equation for the front}
In view of Lemma \ref{B-invert} we may rewrite Equation \eqref{eq-4-order} as:
\begin{eqnarray*}
\varphi_t={\mathscr B}^{-1}{\mathscr L}(\varphi)+{\mathscr B}^{-1}{\mathscr F}((\varphi_y)^2)
\end{eqnarray*}
or, equivalently, as
\begin{equation}
\varphi_t={\mathscr A}(\varphi)
+{\mathscr M}((\varphi_y)^2).
\label{third-order-varphi}
\end{equation}
We emphasize that Equation \eqref{third-order-varphi} is a pseudo-differential, fully nonlinear equation of the third-order,
since the pseudo-differential operators ${\mathscr A}$ and ${\mathscr M}$  have, respectively, symbol
\begin{align*}
a_k=\frac{(U^2-X_k^2)(X_k^2-\gamma U^2)}{4UX_k  },\qquad\;\,
m_k=\frac{X^3_k-3UX_k^2-4\gamma U^2X_k +4\gamma U^3}{4UX_k   },
\end{align*}
for $k\ge 0$. Clearly, any smooth enough solution to \eqref{eq-4-order} solves
\eqref{third-order-varphi} as well.

The following result is crucial for the rest of the paper.

\begin{thm}
\label{prop-symb-1}
The following properties are satisfied.
\begin{enumerate}[\rm (i)]
\item
The realization $A:H^3_{\sharp}\to L^2$ of ${\mathscr A}$ is a sectorial operator and the sequence $(a_k)$ constitutes its spectrum $\sigma(A)$.
In particular, $0$ is a simple
eigenvalue of $A$ and the spectral projection $\Pi$ associated with $0$ is given by
\begin{eqnarray*}
\Pi(\psi)=\frac{1}{\ell}\int_{-\frac{\ell}{2}}^{\frac{\ell}{2}}\psi(y)dy,\qquad\;\,\psi\in L^2.
\end{eqnarray*}
Finally, $\sigma(A)\setminus\{0\}$ is contained in the left half-plane $\{\lambda\in\C: {\rm Re}\lambda<0\}$
if and only if $\gamma<\gamma_c$  .
\item
The realization $M:H^2_{\sharp}\to L^2$ of the operator ${\mathscr M}$ is bounded.
\end{enumerate}
\end{thm}
\begin{proof}
(i).
Let us split
\begin{align*}
a_k&=-\frac{2\lambda_k^{\frac{3}{2}}}{U}+\frac{-4\lambda_k^2+4\lambda_k^2
\sqrt{\frac{U^2}{4\lambda_k}+1}- \lambda_k U^2+\lambda_k \gamma U^2}{U\sqrt{U^2+4\lambda_k}}
=:-\frac{2\lambda_k^{\frac{3}{2}}}{U} +a_{1,k},
\end{align*}
for any $k\ge 0$. Since
\begin{eqnarray*}
a_{1,k}\sim \frac{1}{4}\sqrt{\lambda_k}(2\gamma-1)U,
\end{eqnarray*}
as $k\to +\infty$, if $\gamma\neq 1/2$, we can infer that the realization $A$ of operator
${\mathscr A}$ in $H^3_{\sharp}$ is well defined. Moreover, since $A$ splits into the sum of
two operators $A_0$ (whose symbol is ($-2\lambda_k^{\frac{3}{2}}U^{-1}$)) and $A_1$, which is a
nice perturbation of $A_0$ (being a bounded operator in $H^1_{\sharp}$, which is
an intermediate space of class $J_{1/3}$ between $L^2$ and $H^3_{\sharp}$),
in view of \cite[Prop. 2.4.1(i)]{lunardi} it is enough to prove that $A_0$ is a sectorial operator.
But this follows immediately from general abstract results (see e.g., \cite[Chpt. 3]{haase2006functional}), or a direct computation. Indeed, if $\lambda$ has positive real part, then the equation $\lambda u-A_0u=f$ has, for any $f\in L^2$,
the unique solution
\begin{eqnarray*}
u=R(\lambda,A_0)f=U\sum_{k=0}^{+\infty}\frac{\widehat f(k)}{\lambda U+2\lambda_k^{3/2}}
\end{eqnarray*}
and
\begin{eqnarray*}
|R(\lambda,A_0)f|^2_2=U^2\sum_{k=0}^{+\infty}\frac{|\widehat f(k)|^2}{|\lambda U+2\lambda_k^{3/2}|^2}
\le \frac{1}{|\lambda|^2}\sum_{k=0}^{+\infty}|\widehat f(k)|^2=\frac{1}{|\lambda|^2}|f|_2^2.
\end{eqnarray*}
Proposition 2.1.1 in \cite{lunardi} yields the sectoriality of  $A_0$.

Next we compute the spectrum of the operator $A$.
Since $H^3_{\sharp}$ is compactly embedded into $L^2$, $\sigma(A)$ consists
of eigenvalues only. We claim that $\sigma(A)$ consists of the elements of the sequence $(a_k)$.
Indeed writing the eigenvalue equation in
Fourier variable, we get the infinitely many equations
\begin{equation}
\lambda\widehat\psi(k)-a_k\widehat\psi(k)=0,\qquad\;\,k\ge 0,
\label{L-3}
\end{equation}
which should be satisfied by the pair $\lambda$ (the eigenvalue) and $\psi$ (the eigenfunction).
It is clear that this system of infinitely many equations admits a non identically solution $(\widehat\psi(k))$  vanishing
if and only if $\lambda$ equals one of the elements of the sequence.
The set equality $\sigma(A)=\{a_k: k\ge 0\}$ is thus proved.

Since the sequence $(a_k)$ diverges to $-\infty$ as $k\to +\infty$, all the eigenvalues of $A$ are isolated. In particular,
$0$ is isolated and, again from formula \eqref{L-3}, we easily see that the eigenspace associated with the eigenvalue $\lambda=0$ is
one-dimensional. To conclude that $\lambda=0$ is simple, in view of
\cite[Props. A.1.2 \& A.2.1]{lunardi}, it suffices to prove that it is a simple pole of the resolvent operator.
In such a case the associated spectral projection is the residual at $\lambda=0$ of $R(\cdot,A)$.

Clearly, for any $\lambda\not\in\sigma(A)$,
\begin{eqnarray*}
R(\lambda,A)\zeta=\sum_{k=0}^{+\infty}\frac{1}{\lambda-a_k}\widehat \zeta(k)w_k,
\end{eqnarray*}
for any $\zeta\in L^2$.
Hence,
\begin{eqnarray*}
\lambda R(\lambda,A)\zeta=\widehat \psi(0)w_0+\sum_{k=1}^{+\infty}\frac{\lambda}{\lambda-a_k}\widehat\zeta(k)w_k=:\Pi\zeta
+R_1(\lambda)\zeta.
\end{eqnarray*}
Since $\lambda\neq a_k$ for any $k\ge 1$, and $a_k\to -\infty$ as $k\to +\infty$, there exists
a neighborhood of $\lambda=0$ in which the ratio $\left |\lambda\right |/\left |\lambda-a_k\right |$
is bounded, uniformly with respect to $k\ge 1$. As a byproduct, in such a neighborhood of $\lambda=0$,
the mapping $\lambda\mapsto R_1(\lambda)$ is bounded with values in $L(L^2)$. This shows that $\lambda=0$ is a simple
pole of operator $A$.

To conclude the proof of point (i), let us determine the values of $\gamma$ such that $\sigma(A)\setminus\{0\}$
does not contain nonnegative elements. For this purpose, it suffices to observe that $a_k<0$ for any
$k\ge 1$ if and only if $4\lambda_k+U^2-\gamma U^2>0$ for such $k$'s, which is equivalent to
$4\lambda_1+U^2-\gamma U^2>0$ since $(\lambda_k)$ is nondecreasing sequence.
Hence, the condition for $\sigma(A)\setminus\{0\}$ be contained in $(-\infty,0)$ is $\gamma<\gamma_c$, where
\begin{equation}
\gamma_c=1+\frac{16\pi^2}{\ell^2U^2}.
\label{eq-c}
\end{equation}

(ii). As in the proof of Lemma \eqref{B-invert}, it suffices to observe that
$m_k\sim\lambda_k U^{-1}$ as $k\to +\infty$.
\end{proof}

The linearized stability principle (see e.g., \cite[Sect. 9.1.1]{lunardi}) and the results in
Theorem \ref{prop-symb-1} yield to the following stability analysis.

\begin{coro}
\label{cor:stab} Let $\gamma_c$ be given by \eqref{eq-c}.
\begin{enumerate}[\rm (a)]
\item
If $\gamma<\gamma_c$, then the null solution to Equation
\eqref{third-order-varphi} is {\rm (}orbitally{\rm)} stable, with asymptotic
phase, with respect to sufficiently smooth and small perturbations.
\item
If $\gamma>\gamma_c$, then the null solution to Equation
\eqref{third-order-varphi} is unstable.
\end{enumerate}
\end{coro}

\section{Rigorous asymptotic derivation of the K-S equation}
\setcounter{equation}{0} \label{asymptotical}
The last question that we address is the link between \eqref{third-order-varphi} and Equation \eqref{eqn:KS}.
As in Section \ref{ansatz}, we consider the small perturbation parameter $\e >0$ defined by (see \eqref{perturbation})
\begin{equation*}
\gamma=1+\e.
\end{equation*}
Moreover, we perform the same change of dependent and independent variables as in \eqref{rescaled var},
namely:
\begin{eqnarray*}
t= \frac{\tau}{\varepsilon^2 U^2},\qquad\;\, y = \frac{\eta}{\sqrt{\varepsilon} U}, \qquad\;\,
\varphi = \frac{\varepsilon}{U} \psi.
\end{eqnarray*}
The key-idea is to link the small positive parameter $\e$ and the width of the strip,
which will blow up as $\e \to 0$.
For $\ell_0 >0$ fixed, we take $\ell$ of the form:
\begin{equation*}
\ell_{\varepsilon}= \frac{\ell_0}{\sqrt{\varepsilon}U},
\end{equation*}
hence $\gamma_c$ (see \eqref{eq-c})
converges to $1$ as $\e \to 0$.

In view of Corollary \ref{cor:stab}, in order to avoid a trivial dynamics, we assume that $\gamma_c>1$. This means that
we take  the bifurcation parameter $\ell_0$ larger than $4\pi$ and obtain that $\gamma_c \in (1,1+\varepsilon)$.

In the new variables ${\mathscr B}$ is replaced by the operator ${\mathscr B}_{\e}=U^{2} \sqrt{ (I-4\varepsilon D_{ \eta\eta})}$.
Lemma \ref{B-invert} applies to this operator and guarantees that, for any fixed $\e>0$,
the realization $B_{\e}:H^1_{\sharp}\to L^2$ of ${\mathscr B}_{\e}$ is bounded.
However, the perturbation is clearly singular as $\e\to 0$, since obviously $B_\e \to U^{2}I$. Therefore,
it is hopeless to take the limit $\e \to 0$ in the third-order equation \eqref{third-order-varphi},
whose behaviour is clearly singular as $\e \to 0$. Fortunately, the fourth-order equation
\eqref{eq-4-order} is more friendly, since, after division by $\e^3$ and $U^3$, it comes:
\begin{align}\label{perturbed-4th-order}
\frac{\partial}{\partial\tau} &\left (\sqrt{I-4\e D_{\eta\eta}} \right)\psi
=-4D_{\eta\eta\eta\eta}\psi-D_{\eta\eta}\psi \nonumber\\[1mm]
&+\frac{1}{4}\left\{(I-4\e D_{\eta\eta})^{\frac{3}{2}}-3(I-4\e
D_{\eta\eta})-4(1+\e)\left (\sqrt{I-4\e D_{\eta\eta}}-I\right )
\right\}(D_{\eta}\psi)^2,
\end{align}
which is the perturbed equation we are going to study, with periodic boundary conditions at
$\eta= \pm \ell_0/2$.

Mimicking \eqref{eq-4-order}, we rewrite \eqref{perturbed-4th-order} in the abstract way:
\begin{equation}\label{abstract-perturbed-4th-order}
\frac{d}{d\tau}{\mathscr B}_{\e}\psi = {\mathscr L}\psi  + {\mathscr F}_{\e}((\psi_{\eta})^2),
\end{equation}
where  the symbols of the operators ${\mathscr B}_{\e}$, ${\mathscr L}$ and ${\mathscr F}_{\e}$ are respectively:
\begin{align*}
&b_{\e,k}=  X_{\e,k},\\
&s_k=-\lambda_k(4\lambda_k-1),\\
&f_{\e,k}=\frac{1}{4}(X^3_{\e,k}-3X_{\e,k}^2-4(1+\e) X_{\e,k}+4+4\e),
\end{align*}
for any $k\ge 0$ and
\begin{eqnarray*}
X_{\varepsilon,k}=\sqrt{1+4\varepsilon\lambda_k},\qquad\;\,k\ge 0
\end{eqnarray*}

Writing \eqref{abstract-perturbed-4th-order} in discrete Fourier variable gives the infinitely many equations
\begin{eqnarray*}
b_{\e,k}\widehat\psi_{\tau}(\tau,k)=-\lambda_k(4\lambda_k-1)\widehat\psi(\tau,k)+f_{\e,k}\widehat{(\psi_{\eta})^2}(\tau,k),
\end{eqnarray*}
for any $k\ge 0$.  Note that the leading terms (namely at order $0$ in $\e$) of $b_{\e,k}$ and $f_{\e,k}$ are
$1$ and $-1/2$, respectively.

Fix $T>0$. For $\Phi_0\in H^m_{\sharp}$ ($m\ge 4$) the Cauchy problem
\begin{eqnarray*}
\left\{
\begin{array}{lll}
\Phi_\tau(\tau,\eta)= -4\Phi_{\eta\eta\eta\eta}(\tau,\eta)-
\Phi_{\eta\eta}(\tau,\eta)
- \frac{1}{2}(\Phi_\eta(\tau,\eta))^2, &\tau\ge 0, & |\eta|\le\frac{\ell_0}{2},\\[2mm]
D_{\eta}^k\Phi(\tau,-\ell_0/2)=D_{\eta}^k\Phi(\tau,\ell_0/2), &\tau\ge 0, & k\ge 0,\\[2mm]
\Phi(0,\eta)=\Phi_0(\eta), &&|\eta|\le \frac{\ell_0}{2}.
\end{array}
\right.
\end{eqnarray*}
admits a unique solution $\Phi\in C([0,T];H^m_{\sharp})$ such that
$\Phi_{\tau}\in C([0,T];H^{m-4}_{\sharp})$ (see e.g., \cite[App. B]{BHLS10}).

Through $\Phi$ we split $\psi= \Phi+ \e  \rho_\e$.
For simplicity, we take zero as the initial condition for $\rho_{\varepsilon}$
and, to avoid cumbersome
notation, in the sequel we usely write $\rho$ for $\rho_\e$.

If $\psi$ solves \eqref{abstract-perturbed-4th-order}, then
\begin{align}
\frac{\partial}{\partial\tau}{\mathscr B}_{\e}(\rho) +{\mathscr
H}_{\e}(\Phi_{\tau}) ={\mathscr L}(\rho)+{\mathscr
M}_{\e}((\Phi_{\eta})^2)+\e{\mathscr
F}_{\e}((\rho_{\eta})^2)+2{\mathscr F}_{\e}(\Phi_{\eta}\rho_{\eta}),
\label{pb-fin-ve}
\end{align}
where the symbols of the operators
${\mathscr H}_{\e}$ and ${\mathscr M}_{\e}$ are
\begin{align*}
h_{\e,k}=\frac{1}{\e}(X_{\e,k}-1),\qquad\;\,
m_{\e,k}=\frac{1}{4\e}(X^3_{\e,k}-3X_{\e,k}^2-4(1+\e)
X_{\e,k}+6+4\e),
\end{align*}
for any $k\ge 0$.

\begin{prop}
\label{prop-H} There exists a positive
constant $C_*$ such that the following properties are satisfied for any $\e\in (0,1]$:
\begin{enumerate}[\rm (a)]
\item
for any $s=2,3,\ldots$, the operators ${\mathscr B}_{\e}$ and
${\mathscr H}_{\e}$ admit bounded realizations $B_{\e}$ and
$H_{\e}$, respectively, mapping $H^s_{\sharp}$ into
$H^{s-2}_{\sharp}$. Moreover,
\begin{eqnarray*}
\|B_{\e}\|_{L(H^s_{\sharp},H^{s-2}_{\sharp})}
+\|H_{\e}\|_{L(H^s_{\sharp},H^{s-2}_{\sharp})}\le C_*.
\end{eqnarray*}
Finally, the operator
$B_{\e}$ is invertible from $H^s_{\sharp}$ to
$H^{s-2}_{\sharp}$;
\item
for any $s=3,4,\ldots$, the operators ${\mathscr F}_{\e}$ and ${\mathscr
M}_{\e}$ admit bounded realizations $F_{\e}$ and $M_{\e}$, respectively, mapping $H^s$ into $H^{s-3}$. Moreover,
\begin{eqnarray*}
\|F_{\e}\|_{L(H^s_{\sharp},H^{s-3}_{\sharp})}+\|M_{\e}\|_{L(H^s_{\sharp},H^{s-3}_{\sharp})}\le C_*.
\end{eqnarray*}
\end{enumerate}
\end{prop}

\begin{proof}
The statement follows from an analysis of the symbols of the operators ${\mathscr B}_{\e}$, ${\mathscr F}_{\e}$,
${\mathscr H}_{\e}$ and ${\mathscr M}_{\e}$. Without much effort one can show that
\begin{eqnarray*}
|h_{\e,k}|\le 4\lambda_k,\qquad\;\,
|m_{\e,k}|\le 2\lambda_k^{\frac{3}{2}}
+25\lambda_k,
\end{eqnarray*}
for any $k\ge 0$ and any $\e\in (0,1]$. These estimates combined with
the formulas $0\neq b_{\e,k}=\e h_{\e,k}+1$
and $f_k=\e m_{\e,k}-1/2$, for any
$k\ge 0$ and any $\varepsilon\in (0,1]$, yield the assertion.
\end{proof}

Instead of studying Equation \eqref{eq-4-order}, we find it much more convenient to deal with the equation satisfied by $\zeta:=\rho_{\eta}$, i.e.,
\begin{equation}
\frac{\partial}{\partial\tau}{\mathscr B}_{\e}(\zeta) +{\mathscr H}_{\e}(\Psi_{\tau}) ={\mathscr L}(\zeta)+{\mathscr M}_{\e}((\Psi^2)_{\eta})+\e {\mathscr F}_{\e}((\zeta^2)_{\eta})+2{\mathscr F}_{\e}((\Psi\zeta)_{\eta}),
\label{eq-final-ve-1}
\end{equation}
which we couple with the initial condition $\zeta(0,\cdot)=0$. Here, $\Psi=\Phi_{\eta}$.

\subsection{A priori estimates}\label{Formal-estimates}

For any $n=0,1,2,\ldots$ and any $T>0$, we set
\begin{eqnarray*}
X_n(T)=\left\{\zeta\in C([0,T];H^{4\vee 2n}_{\sharp})\cap C^1([0,T];L^2):\zeta_{\tau}\in C([0,T];H^{2\vee (n+1)}_{\sharp})\right\},
\end{eqnarray*}
where $a\vee b:=\max\{a,b\}$.

For any $\e>0$, we introduce in $H^{1/2}_{\sharp}$ the norm
$\|\zeta\|_{\frac{1}{2},\e}^2=\sum_{k=0}^{+\infty} \sqrt{1+4\e \lambda_k}\, |\widehat\zeta(k)|^2$ for any $\zeta\in H^{1/2}_{\sharp}$.
Note that, for any fixed $\e>0$, $\|\cdot\|_{\frac{1}{2},\e}$ is a norm, equivalent to
the usual norm in $H^{1/2}_{\sharp}$.

The main result of this subsection is contained in the following theorem, where we set $\Psi_0=(\Phi_0)_{\eta}$.

\begin{thm}
\label{cor-apriori-estim} Fix an integer $n\geq 0$ and $T>0$.
Further, suppose that $\Psi_0\in H^{n+6}_{\sharp}$.
Then, there exist
$\varepsilon_1=\e_1(n,T)\in (0,1)$ and $K_n=K_n(n,T)>0$ such that, if
$\zeta\in X_n(T_1)$ is a solution on the time interval $[0,T_1]$ of
Equation \eqref{eq-final-ve-1} for some $T_1\le T$, then
\begin{align}\label{main-estimate}
 \sup_{\tau\in [0,T_1]}\|D^{n}_{\eta}\zeta(\tau,\cdot)\|_{\frac{1}{2},\e }^2\le K_{n},
\end{align}
whenever $0<\varepsilon \leq\varepsilon_1$.
\end{thm}

Note that the assumptions on $\Psi_0$ guarantee that $\Psi$ belongs to $C([0,T];H^{n+4}_{\sharp})\cap C^1([0,T];H^{n+2}_{\sharp})$.

The proof of Theorem \ref{cor-apriori-estim} heavily relies on the following lemma.

\begin{lem}
\label{lem-5} Let $A_0$, $c_0$, $c_1$, $c_2$, $c_3$, $\e$, $T_0$, $T_1$ be positive constants with $T_1<T_0$.
Further, let $f_{\e},A_{\e}:[0,T_1]\to\mathbb R$
be a positive continuous function and a positive continuously differentiable function such that
\begin{align*}
\left\{
\begin{array}{ll}
A_{\e}'(\tau)+(c_0-\e^2 (A_{\e}(\tau))^2)f_{\e}(\tau)\le c_1+c_2A_{\e}(\tau)+c_3\e (A_{\e}(\tau))^2, & \tau\in [0,T_1],\\[2mm]
A_{\e}(0)=0.
\end{array}
\right.
\end{align*}
Then, there exist $\e_1=\e_1(T_0)\in (0,1)$ and
a constant $K=K(T_0)$ such that $A_{\e}(\tau)\le K$ for any
$\tau\in [0,T_1]$ and any $\e\in (0,\e_1]$.
\end{lem}

\begin{proof}
When $f_{\e}$ identically vanishes, the proof follows from \cite[Lemma 3.1]{BFHLS}, which shows that we can
take $\e_1(T_0)=3c_2^2/(16c_1c_3(e^{c_2T_0}-1))$ and $K\le 4c_1e^{c_2T_0}/(3c_2)$.

Let us now consider the general case when $f_{\e}$ does not
identically vanish in $[0,T_1]$. We fix $\e_0=\e_0(T_0)\le3c_2^2/(16c_1c_3(e^{c_2T_0}-1))$ such
that $9c_0c_2^2-12c_1c_2e^{c_2T_0}\e_0-16c_1^2e^{2c_2T_0}\e_0^2>0$, and $\e\in (0,\e_0]$.
We claim that $c_0-\e^2 (A_{\e}(\tau))^2>0$ for any $\tau\in [0,T_1]$.

Let $(0,T_{\ve})$ be the largest interval (possibly depending on $\varepsilon$) where
$c_0-\e A_{\e}-\e^2 (A_{\e})^2$ is positive. The existence of this interval is clear since $A_{\e}$
vanishes at $0$. The positivity of $c_0-\e^2 (A_{\e})^2$ in $(0,T_{\e})$ shows that
$A_{\e}'\le c_1+c_2A_{\e}+c_3\e A_{\e}^2$ in such an interval and, from the above result, we can infer that
$A_{\e}(\tau)\le (4c_1e^{c_2T_0})/(3c_2)$ for any $\tau\in [0,T_{\e}]$, so that
$c_0-\e^2 (A_{\e}(T_{\e}))^2>0$. By the definition of $T_{\e}$, this clearly implies that $T_{\e}=T_1$.
\end{proof}

\begin{proof}[Proof of Theorem $\ref{cor-apriori-estim}$]
Throughout the proof we assume that $T_1\le T$ is fixed and $\e$ and $\tau$ are arbitrarily fixed in $(0,1]$ and in $[0,T_1]$, respectively.
Moreover,  to avoid cumbersome notation, we denote by $c$ almost all the constants appearing in the estimates.
Hence, the exact value of $c$ may change from line to line, but we do not need to follow the constants throughout the estimates.
We just need to stress how the estimates depend on $\e$. As a matter of fact, all the $c$'s are independent not only of $\e$ but also of $\tau$, $\Psi$ and $\zeta$. On the contrary, they may depend on $n$ (and, actually, in most the cases they do).
Finally, we denote by $K(\Psi)$ a constant which may depend on $n$ and also on $\Psi$. As above $K(\Psi)$ may vary from estimate to estimate.

The first step of the proof consists in multiplying both sides of Equation \eqref{eq-final-ve-1} by
$(-1)^n D^{2n}_{\eta}\zeta$ and integrating by parts over
$(-\ell_0/2,\ell_0/2)$. This yields to the equation
\begin{align}
&\int_{-\frac{\ell_0}{2}}^{\frac{\ell_0}{2}}B_{\e}(\zeta_{\tau}(\tau,\cdot))(-1)^n D^{2n}_{\eta}\zeta(\tau,\cdot)d\eta
+4\int_{\frac{\ell_0}{2}}^{\frac{\ell_0}{2}}|D^{n+2}_{\eta}\zeta(\tau,\cdot)|^2d\eta\\
&-\int_{\frac{\ell_0}{2}}^{\frac{\ell_0}{2}}|D^{n+1}_{\eta}\zeta(\tau,\cdot)|^2d\eta
\notag \\
=&  -\int_{-\frac{\ell_0}{2}}^{\frac{\ell_0}{2}}\left
(H_{\e}(\Psi_{\tau}(\tau,\cdot))
-M_{\e}((\Psi^2)_{\eta}(\tau,\cdot))\right ) (-1)^n D^{2n}_{\eta}\zeta(\tau,\cdot)d\eta\notag\\
&+\e\int_{-\frac{\ell_0}{2}}^{\frac{\ell_0}{2}}F_{\e}((\zeta^2)_{\eta}(\tau,\cdot))
(-1)^n D^{2n}_{\eta}\zeta(\tau,\cdot)d\eta\\
&+2\int_{-\frac{\ell_0}{2}}^{\frac{\ell_0}{2}}F_{\e}((\Psi\zeta)_{\eta}(\tau,\cdot))
(-1)^n D^{2n}_{\eta}\zeta(\tau,\cdot)d\eta.
\label{variational}
\end{align}

Using Parseval's formula and the definition
of the symbol $b_{\e,k}$, one can easily show that
\begin{align}
\int_{-\frac{\ell_0}{2}}^{\frac{\ell_0}{2}}B_{\e}(\zeta_{\tau}(\tau,\cdot)) (-1)^n D^{2n}_{\eta}\zeta(\tau,\cdot)d\eta= \frac{1}{2}\frac{d}{d\tau}\|D^n_{\eta}\zeta(\tau,\cdot)\|_{\frac{1}{2},\e}^2.
\label{LLLL}
\end{align}

We now deal with the other terms in \eqref{variational}.
Integrating $n$-times by parts and, then, using  Poincar\'e-Wirtinger and Cauchy-Schwarz
inequalities, jointly with Proposition \ref{prop-H}, it is not difficult to show that
\begin{align}
&\left |\int_{-\frac{\ell_0}{2}}^{\frac{\ell_0}{2}}\left
(H_{\e}(\Psi_{\tau}(\tau,\eta))-M_{\e}((\Psi^2)_{\eta}(\tau,\cdot))\right
)(-1)^n D^{2n}_{\eta}\zeta(\tau,\cdot)d\eta\right | \leq K(\Psi)+|D^n_{\eta}
\zeta(\tau,\cdot)|_2^2, \label{estim-bad-1}
\end{align}
for any $\zeta\in X_n(T_1)$.

Estimating the other two integral terms in the right-hand side of \eqref{variational} demands some more effort.
The starting point is the following
estimate
\begin{align}
&\left |\int_{-\frac{\ell_0}{2}}^{\frac{\ell_0}{2}}F_{\e}(\chi_{\eta}(\tau,\cdot))(-1)^n D^{2n}_{\eta}\zeta(\tau,\cdot)
d\eta\right |\notag\\[1mm]
\le &c\e^{\frac{3}{2}}|D^{n+2}_{\eta}\chi(\tau,\cdot)|_2|D^{n+2}_{\eta} \zeta(\tau,\cdot)|_2
+c\e|D^{n+2}_{\eta}\chi(\tau,\cdot)|_2|D^{n+1}_{\eta}\zeta(\tau,\cdot)|_2\notag\\[1mm]
&+c\sqrt{\e}|D^n_{\eta}\chi(\tau,\cdot)|_2|D^{n+2}_{\eta}\zeta(\tau,\cdot)|_2
+c|D^n_{\eta}\chi(\tau,\cdot)|_2|D^{n+1}_{\eta}\zeta(\tau,\cdot)|_2,
\label{estim-6}
\end{align}
which holds true for any $\chi\in C([0,T_1];H^{4\vee 2n}_{\sharp})$.
Such a formula follows observing that
\begin{align*}
&\left
|\int_{-\frac{\ell_0}{2}}^{\frac{\ell_0}{2}}F_{\e}(\chi_{\eta}(\tau,\cdot))(-1)^n D^{2n}_{\eta}\zeta(\tau,\cdot)
d\eta\right |\\
\le&\sum_{k=0}^{+\infty}\lambda_k^n|f_{\e,k}||\widehat{\chi_{\eta}}(\tau,k)||\widehat\zeta(\tau,k)|\\
\le& c\sum_{k=0}^{+\infty}\lambda_k^n\left (\e^{\frac{3}{2}}\lambda_k^{\frac{3}{2}}+\e\lambda_k+\e^{\frac{1}{2}}\lambda_k^{\frac{1}{2}}+1\right )
|\widehat{\chi_{\eta}}(\tau,k)||\widehat\zeta(\tau,k)|,
\end{align*}
and, then, using Young inequality, to estimate the terms in the round brackets.

Now, we plug $\chi=\zeta^2$ into \eqref{estim-6} and use the estimates
\begin{align}
&|D^{n+2}_{\eta}(\zeta(\tau,\cdot))^2|_2 \leq c (|D^{n+2}_{\eta}\zeta(\tau,\cdot)|_2 |D^{n}_{\eta}\zeta(\tau,\cdot)|_2 + |D^{n+1}_{\eta}\zeta(\tau,\cdot)|_2 ^2),\label{estim-7}
\\[1mm]
&|D^{n}_{\eta}(\zeta(\tau,\cdot))^2|_2 \leq c|D^{n}_{\eta}\zeta(\tau,\cdot)|_2^2,
\label{estim-8}
\end{align}
(which can be obtained using the Poincar\'e-Wirtinger inequality and Leibniz formula),
the Cauchy-Schwarz inequality and, again, the Poincar\'e-Wirtinger inequality, to obtain
\begin{align}
&\left |\int_{-\frac{\ell_0}{2}}^{\frac{\ell_0}{2}}F_{\e}((\zeta^2)_{\eta}(\tau,\cdot))(-1)^n D^{2n}_{\eta}\zeta(\tau,\cdot)d\eta\right |\notag\\[1mm]
\le & c\e^{\frac{3}{2}}|D^n_{\eta}\zeta(\tau,\cdot)|_2|D^{n+2}_{\eta}\zeta(\tau,\cdot)|_2^2
+c\e^{\frac{3}{2}}|D^{n+1}_{\eta}\zeta(\tau,\cdot)|_2^2|D^{n+2}_{\eta}\zeta(\tau,\cdot)|_2\notag\\[1mm]
&+c\e|D^n_{\eta}\zeta_{\eta}(\tau,\cdot)|_2|D^{n+1}_{\eta}\zeta(\tau,\cdot)|_2|D^{n+2}_{\eta}\zeta(\tau,\cdot)|_2
+c\e|D^{n+1}_{\eta}\zeta(\tau,\cdot)|_2^3\notag\\[1mm]
&+c\e^{\frac{1}{2}}(1+\e)|D^n_{\eta}\zeta(\tau,\cdot)|_2^2|D^{n+2}_{\eta}\zeta(\tau,\cdot)|_2
+c|D^n_{\eta}\zeta(\tau,\cdot)|_2^2|D^{n+1}_{\eta}\zeta(\tau,\cdot)|_2\notag\\[1mm]
\le &c\e^{\frac{3}{2}}|D^n_{\eta}\zeta(\tau,\cdot)|_2|D^{n+2}_{\eta}\zeta(\tau,\cdot)|_2^2
+c\e^{\frac{3}{2}}|D^{n+1}_{\eta}\zeta(\tau,\cdot)|_2^2|D^{n+2}_{\eta}\zeta(\tau,\cdot)|_2\notag\\[1mm]
&+c\e|D^{n+1}_{\eta}\zeta(\tau,\cdot)|_2^2|D^{n+2}_{\eta}\zeta(\tau,\cdot)|_2
+c\e|D^{n+1}_{\eta}\zeta(\tau,\cdot)|_2^2|D^{n+2}_{\eta}\zeta(\tau,\cdot)|_2\notag\\[1mm]
&+c\e^{\frac{1}{2}}|D^n_{\eta}\zeta(\tau,\cdot)|_2^2|D^{n+2}_{\eta}\zeta(\tau,\cdot)|_2
+c|D^n_{\eta}\zeta(\tau,\cdot)|_2^2|D^{n+2}_{\eta}\zeta(\tau,\cdot)|_2\notag\\[1mm]
\le & c\e^{\frac{3}{2}}|D^n_{\eta}\zeta(\tau,\cdot)|_2|D^{n+2}_{\eta}\zeta(\tau,\cdot)|_2^2
+c\e^2|D^{n+1}_{\eta}\zeta(\tau,\cdot)|_2^4+c\e|D^{n+2}_{\eta}\zeta(\tau,\cdot)|_2^2\notag\\
&+c|D^n_{\eta}\zeta(\tau,\cdot)|_2^4+c|D^{n+2}_{\eta}\zeta(\tau,\cdot)|_2^2\notag\\[1mm]
\le & c\e^{\frac{3}{2}}|D^n_{\eta}\zeta(\tau,\cdot)|_2|D^{n+2}_{\eta}\zeta(\tau,\cdot)|_2^2
+c\e^2|D^{n+1}_{\eta}\zeta(\tau,\cdot)|_2^4+c|D^{n+2}_{\eta}\zeta(\tau,\cdot)|_2^2+c|D^n_{\eta}\zeta(\tau,\cdot)|_2^4.
\label{estim-bad-2}
\end{align}

In the similar way, using the estimate
$|D^m_{\eta}(\Psi\zeta)(\tau,\cdot)|_2\le c|D^m_{\eta}\zeta(\tau,\cdot)|_2|D^m_{\eta}\Psi(\tau,\cdot)|_2$, (with $m\in\{n,n+2\}$)
in place of \eqref{estim-7} and \eqref{estim-8},
from \eqref{estim-6} we get
\begin{align}
&\left
|\int_{-\frac{\ell_0}{2}}^{\frac{\ell_0}{2}}F_{\e}((\Psi\zeta)_{\eta}(\tau,\cdot))(-1)^n D^{2n}_{\eta}\zeta(\tau,\cdot)
d\eta\right |\notag\\[1mm]
\le & c\e^{\frac{3}{2}}|D^{n+2}_{\eta}\Psi(\tau,\cdot)|_2|D^{n+2}_{\eta}\zeta(\tau,\cdot)|_2^2
+c\e |D^{n+2}_{\eta}\Psi(\tau,\cdot)|_2|D^{n+1}_{\eta}\zeta(\tau,\cdot)|_2^2\notag\\[1mm]
&+c\e |D^{n+2}_{\eta}\Psi(\tau,\cdot)|_2|D^{n+2}_{\eta}\zeta(\tau,\cdot)|_2^2
+c|D^n_{\eta}\Psi(\tau,\cdot)|_2|D^n_{\eta}\zeta(\tau,\cdot)|_2^2\notag\\[1mm]
&+c\e |D^n_{\eta}\Psi(\tau,\cdot)|_2|D^{n+2}_{\eta}\zeta(\tau,\cdot)|_2^2
+c\delta^{-1}|D^n_{\eta}\Psi(\tau,\cdot)|_2^2|D^n_{\eta}\zeta(\tau,\cdot)|_2^2
+c\delta|D^{n+2}_{\eta}\zeta(\tau,\cdot)|_2^2,
\label{LLL}
\end{align}
for any $\delta>0$. We just mention the inequality
\begin{align*}
&|D^{n}\Psi(\tau,\cdot)|_2|D^n\zeta(\tau,\cdot)|_2|D^{n+1}\zeta(\tau,\cdot)|_2\\[1mm]
\le &
c\delta^{-1}|D^{n}\Psi(\tau,\cdot)|_2^2|D^n\zeta(\tau,\cdot)|_2^2
+\delta |D^{n+1}\zeta(\tau,\cdot)|_2^2,
\end{align*}
obtained by means of Young and Poincar\'e-Wirtinger inequalities, which we use to estimate one of the intermediate terms appearing
in the proof of \eqref{LLL}.

Now, taking $c\delta=5/2$, we get the estimate
\begin{align}
&\left|\int_{-\frac{\ell_0}{2}}^{\frac{\ell_0}{2}}F_{\e}((\Psi\zeta)_{\eta}(\tau,\cdot))
(-1)^n D^{2n}_{\eta}\zeta(\tau,\cdot)d\eta\right |\notag\\
\le &K(\Psi)\left (\e|D^{n+2}_{\eta}\zeta(\tau,\cdot)|_2^2+\e|D^{n+1}_{\eta}\zeta(\tau,\cdot)|_2^2
+|D^n_{\eta}\zeta(\tau,\cdot)|_2^2\right )
+\frac{5}{2}|D^{n+2}_{\eta}\zeta(\tau,\cdot)|_2^2.
\label{estim-bad-3}
\end{align}

From \eqref{variational}, \eqref{LLLL}, \eqref{estim-bad-1}, \eqref{estim-bad-2}, \eqref{estim-bad-3} and
the interpolative inequality
\begin{eqnarray*}
|D^{n+1}_{\eta}\zeta(\tau,\cdot)|_2^2\le |D^n_{\eta}\zeta(\tau,\cdot)|_2^2+\frac{1}{4}|D^{n+2}_{\eta}\zeta(\tau,
\cdot)|_2^2,
\end{eqnarray*}
we can infer that
\begin{align*}
&\frac{1}{2}\frac{d}{d\tau}\|D^{n}_{\eta}\zeta(\tau,\cdot)\|_{\frac{1}{2},\e}^2
+(1-\e K(\Psi)
-c\e^{\frac{5}{2}}|D^n_{\eta}\zeta(\tau,\cdot)|_2)|D^{n+2}_{\eta}\zeta(\tau,\cdot)|_2^2\notag\\[1mm]
\le & K(\Psi)
+K(\Psi)|D^n_{\eta}\zeta(\tau,\cdot)|_2^2
+\e K(\Psi)|D^{n+1}_{\eta}\zeta(\tau,\cdot)|_2^2\\[1mm]
&+c\e |D^n_{\eta}\zeta(\tau,\cdot)|_2^4+c\e^3|D^{n+1}_{\eta}\zeta(\tau,\cdot)|_2^4\notag\\[1mm]
\le & K(\Psi)+K(\Psi)|D^n_{\eta}\zeta(\tau,\cdot)|_2^2
+\e K(\Psi)|D^{n+2}_{\eta}\zeta(\tau,\cdot)|_2^2\notag\\[1mm]
&+c\e |D^n_{\eta}\zeta(\tau,\cdot)|_2^4+c\e^3|D^{n}_{\eta}\zeta(\tau,\cdot)|_2^2 |D^{n+2}_{\eta}\zeta(\tau,\cdot)|_2^2,
\end{align*}
which we can rewrite in the form
\begin{align*}
&\frac{d}{d\tau}\|D^{n}_{\eta}\zeta(\tau,\cdot)\|_{\frac{1}{2},\e}^2
+\left (2-\e K(\Psi)
-c\e^2\|D^n_{\eta}\zeta(\tau,\cdot)\|_{\frac{1}{2},\e}^2\right )|D^{n+2}_{\eta}\zeta(\tau,\cdot)|_2^2\notag\\[1mm]
\le & K(\Psi)
+K(\Psi)\|D^n_{\eta}\zeta(\tau,\cdot)\|_{\frac{1}{2},\e}^2
 +c\e\|D^n_{\eta}\zeta(\tau,\cdot)\|_{\frac{1}{2},\e}^4,
\end{align*}
estimating $2\|D^n_{\eta}\zeta(\tau,\cdot)\|_{\frac{1}{2},\e}\le \|D^n_{\eta}\zeta(\tau,\cdot)\|_{\frac{1}{2},\e}^2+1$ and
recalling that $\e\in (0,1]$.

Up to replacing $(0,1]$ by a smaller interval $(0,\e_0]$, we can assume that
$\varepsilon K(\Psi)<1$ for any $\e\in (0,\e_0]$. Hence, applying Lemma \ref{lem-5}, with
\begin{align*}
&c_0=1,\qquad\;\,
c_1=K(\Psi),\qquad\;\,c_2=K(\Psi),\qquad\;\,
c_3=c,\\[1mm]
&A_{\e}(\tau)= \|D^{n}_{\eta}\zeta(\tau,\cdot)\|_{\frac{1}{2},\e}^2,\qquad\;\,
f_{\varepsilon}(\tau)=|D^{n+2}_{\eta}\zeta(\tau,\cdot)|_2^2,\qquad\;\,
\end{align*}
we conclude the proof.
\end{proof}

Now, taking advantage of the previous {\it a priori estimates}, which can be extended also to
variational solutions $\zeta_N$ to \eqref{eq-final-ve-1} belonging to the space spanned by the functions $w_1,\ldots,w_N$ (with
constants independent of $N\in\mathbb N$), and using the classical Faedo-Galerkin method, the following result can be proved.

\begin{thm}\label{zeta}
Fix $T>0$. Then, there exists $\e_0(T)>0$ such that, for any $0<\e \leq \e_0(T)$, Equation \eqref{eq-final-ve-1}
has a unique classical solution $\zeta$ on $[0,T]$, vanishing at $\tau=0$.
\end{thm}

\subsection{Proof of Main Theorem}

Since the unique solution $\zeta$ of Problem \eqref{eq-final-ve-1} is the candidate to be the $\eta$-derivative
of the solution $\rho$ to Problem  \eqref{pb-fin-ve}, $\rho$ should split into the sum
$\rho(\tau,\eta) = ({\mathscr P}(\zeta))(\tau,\eta) + \upsilon(\tau)$ for some scalar valued function $\upsilon$,
where
\begin{eqnarray*}
({\mathscr P}(\zeta))(\tau,\eta)
=\int_{-\frac{\ell_0}{2}}^{\eta}\zeta(s)ds-\frac{1}{2}\int_{-\frac{\ell_0}{2}}^{\frac{\ell_0}{2}}\zeta(s)\left (1-\frac{2s}{\ell_0}\right )ds.
\end{eqnarray*}
Imposing that $\rho$ in the previous form is a solution to \eqref{pb-fin-ve} and projecting along $\Pi(L^2)$, we see that
$\rho$ is a solution to \eqref{pb-fin-ve} if and only if $\upsilon$ solves the
following Cauchy problem:
\begin{eqnarray*}
\left\{
\begin{array}{ll}
\displaystyle\frac{d\upsilon}{d\tau}=-\Pi(H_{\e}(\Phi_{\tau}))-\frac{1}{2}\e\Pi(\zeta^2)-\Pi(\Phi_{\eta}\zeta),\\[3mm]
\upsilon(0)=0.
\end{array}
\right.
\end{eqnarray*}
Since this problem has in fact a unique solution, and ${\mathscr P}(\zeta)+\upsilon$ vanishes at $\tau=0$, we conclude that
problem \eqref{pb-fin-ve} is uniquely solvable.

To complete the proof, we should show that  there exists $M>0$ such that
\begin{equation}
\sup_{{\tau\in [0,T]}\atop{\eta\in
[-\ell_0/2,\ell_0/2]}} |\rho(\tau,\eta)| \leq M,
\label{final-apriori}
\end{equation}
uniformly in $0<\e\leq \e_0(T)$. Once this estimate is proved, coming back from
Problem \eqref{pb-fin-ve} to Equation \eqref{intro-3order}, we see that the latter one has a unique classical solution
$\varphi:[0,\frac{T}{\varepsilon^2U^2}]\times\mathbb R\to\mathbb R$,
which is periodic (with respect to the spatial variable) with period $\ell_{\e}=\ell_0/(\sqrt{\varepsilon}U)$, and satisfies
$\varphi(0,\cdot)=\e U^{-1}\Phi_0(\sqrt{\varepsilon}U\cdot)$,
as well as the estimate
\begin{eqnarray*}
\|\varphi(t,\cdot)-\e U^{-1}\Phi(t\e^2 U^2,\cdot\sqrt{\e}U)\|_{C([-\ell_{\e}/2,\ell_{\e}/2])}\le
\frac{\varepsilon^2 M}{U},\qquad\;\,t\in [0,T_{\e}],
\end{eqnarray*}
as it is claimed.

So, let us prove \eqref{final-apriori}.
For this purpose it is enough to use the a priori estimate \eqref{main-estimate} jointly
with the Poincar\'e-Wirtinger inequality, to estimate $\zeta$, and just \eqref{main-estimate} to estimate $\upsilon$.
This completes the proof of the Main Theorem.

\section{Numerical experiments}\label{numerics}

In this section, we intend to solve numerically Equation \eqref{perturbed-4th-order} for small positive $\e$ and illustrate the convergence to the solution of K-S equation.

In order to reformulate \eqref{perturbed-4th-order} on the interval $[0,2\pi]$ with periodic boundary conditions,
we set $x=\eta/(2\tilde{\ell_0})$, where $\tilde{\ell_0}={\ell_0}/{4\pi}$. It comes:
\begin{align*}
&\frac{\partial}{\partial\tau} \Bigg (\sqrt{I-\frac{\e}{\tilde{\ell_0}^2}D_{xx}} \Bigg )\psi
=-\frac{1}{4\tilde{\ell_0}^4}D_{xxxx}\psi-\frac{1}{4\tilde{\ell_0}^2}D_{xx}\psi\\[1mm]
&+\Bigg\{\Bigg (I-\frac{\e}{\tilde{\ell_0}^2} D_{xx}\Bigg )^{\frac{3}{2}}-3\Bigg (I-\frac{\e}{\tilde{\ell_0}^2}
D_{xx}\Bigg )-4(1+\e)\Bigg (\sqrt{I-\frac{\e}{\tilde{\ell_0}^2}D_{xx}}-I\Bigg )
\Bigg\}\frac{(D_{x}\psi)^2}{16\tilde{\ell_0}^2}.
\end{align*}
Next, we define the bifurcation parameter $\beta=4\tilde{\ell_0}^2$ as in \cite{HN}, \cite{BLSX10}. After multiplication by $\beta^2$, it comes:
\begin{align*}
&\frac{\partial}{\partial\tau} \left (\sqrt{\beta^4-4\e \beta^3 D_{xx}} \right)\psi
=-4 D_{xxxx}\psi-\beta D_{xx}\psi\\[1mm]
&+\frac{\beta}{4}\left\{(I-\frac{4\e}{\beta} D_{xx})^{\frac{3}{2}}-3\left (I-\frac{4\e}{\beta}
D_{xx}\right )-4(1+\e)\left (\sqrt{I-\frac{4\e}{\beta}D_{xx}}-I\right )
\right\}(D_{x}\psi)^2.
\end{align*}
Finally, we rescale the time, setting $t={\tau}/{\beta^2}$:
\begin{align*}
&\frac{\partial}{\partial t} \left (\sqrt{I-\frac{4\e}{\beta} D_{xx}} \right)\psi
=-4 D_{xxxx}\psi-\beta D_{xx}\psi\\[1mm]
&+\frac{\beta}{4}\left\{\left (I-\frac{4\e}{\beta} D_{xx}\right )^{\frac{3}{2}}-3\left (I-\frac{4\e}{\beta}
D_{xx}\right )-4(1+\e)\left (\sqrt{I-\frac{4\e}{\beta}D_{xx}}-I\right )
\right\}(D_{x}\psi)^2,
\end{align*}
and setting $\e' = {\e}/{\beta}$, the prime being omitted hereafter, we obtain:
\begin{align}
&\frac{\partial}{\partial t} \left (\sqrt{I-4\e D_{xx}} \right)\psi \nonumber
=-4 D_{xxxx}\psi-\beta D_{xx}\psi \nonumber \\[1mm]
&+\frac{\beta}{4}\left\{(I-4\e D_{xx})^{\frac{3}{2}}-3(I-4\e
D_{xx})-4(1+\e)\left (\sqrt{I-4\e D_{xx}}-I\right )
\right\}(D_{x}\psi)^2. \label{numerical 1}
\end{align}
The initial condition is given by $\psi(0,\cdot)=\psi_0$,
where $\psi_0$ is periodic with period $2\pi$.
Note that, in contrast to \cite{HN}, \cite{BLSX10}, we do not subtract the drift.

Equation \eqref{numerical 1} reads in discrete Fourier variable:
\begin{align*}
&\frac{\partial}{\partial t} \left (\sqrt{1+4\e k^2} \right)\widehat\psi(t,k)
=-4k^4\widehat\psi(t,k)+\beta k^2\widehat\psi(t,k)\\[1mm]
&+\frac{\beta}{4}\left\{(1+4\e k^2)^{\frac{3}{2}}-3(1+4\e
 k^2)-4(1+\e)\left (\sqrt{1+4\e  k^2}-1\right )
\right\}\widehat{(\psi_x)^2}(t,k).
\end{align*}
We use a backward-Euler schema for the first-order time derivative, treat implicitly all the linear terms and explicitly the nonlinear terms. The implicit treatment of the fourth- and second-order terms reduces the stability
constraint, while the explicit treatment of the nonlinear terms
avoids the expensive process of solving
nonlinear equations at each time step.
For simplicity, in the rest of this section we use the notation $\widehat f_k$ instead of $\widehat f(k)$.
It comes:
\begin{align*}
& \left (\sqrt{ 1+4\e k^2} \right)\frac{\widehat\psi_k^{n+1}-\widehat\psi_k^{n}}{\Delta t}
=-4k^4\widehat\psi_k^{n+1}+\beta k^2\widehat\psi_k^{n+1} \\[1mm]
&+\frac{\beta}{4}\left\{(1+4\e k^2)^{\frac{3}{2}}-3(1+4\e
 k^2)-4(1+\e)\left (\sqrt{1+4\e  k^2}-1\right )
\right\}  \{[(\psi_x)^n]^2\}_k,
\end{align*}
where $\{(\psi_x)^2\}_k$ represents the $k$-th Fourier coefficient of $(\psi_x)^2$.
This method is of the first order with respect to time.
From the previous equation it is easy to compute the $k$-th Fourier coefficient $\widehat\psi_k^{n+1}$. One gets:
\begin{align}
&\widehat\psi_k^{n+1}=\Big ((1+4\e k^2)^{\frac{1}{2}}+4k^4{\Delta t} -\beta k^2{\Delta t}  \Big )^{-1}\notag\\
&\qquad\qquad\times\Big ((1+4\e k^2)^{\frac{1}{2}}\widehat\psi_k^{n}\notag\\
&\qquad\qquad\qquad\;\,
+\frac{\beta{\Delta t}}{4}\left\{(1+4\e k^2)^{\frac{3}{2}}-3( 1+4\e k^2)-4(1+\e)[(1+4\e k^2)^{\frac{1}{2}}-1]\right\}\notag\\
&\qquad\qquad\qquad\qquad\quad\times \big\{[(\psi_x)^n]^2\big\}_k\Big ),
\label{numerical 6}
\end{align}
Practical calculations hold in the spectral space. We use an additional FFT to recover the physical nodal values $\psi_j$ from $\widehat\psi_k$, where $j$ stands for the division node in the physical space.

The numerical tests aim at checking the behavior of the solutions of Equation \eqref{numerical 1} for values of $\varepsilon$ close to $0$, and compare them to the Kuramoto-Sivashinsky equation.
In Figures \ref{fig-eps}-\ref{fig-eps4}, and \ref{fig-eps5}-\ref{fig-eps8}, we plot consecutive front positions computed using Equation \eqref{numerical 6} taking $\beta=10,\,20$ and giving to $\varepsilon$ the following values: $
0.1$, $0.01$, $0.001$, and $0$ (which corresponds to the K-S equation).
\begin{figure}[thbp]
 \centering
\begin{tabular}{c}
\includegraphics[width=9cm,height=3cm,angle=0]{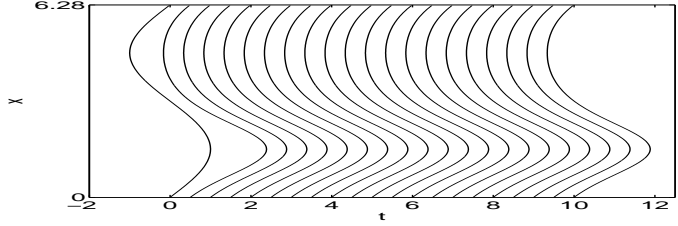}
\end{tabular}
\caption{ Front propagation with $\beta=10$, $\varepsilon=0.1$ and $\psi_0(x)=\sin(x)$.}
\label{fig-eps}
\end{figure}

\begin{figure}[thbp]
 \centering
   \begin{tabular}{c}
\includegraphics[width=9cm,height=3cm,angle=0]{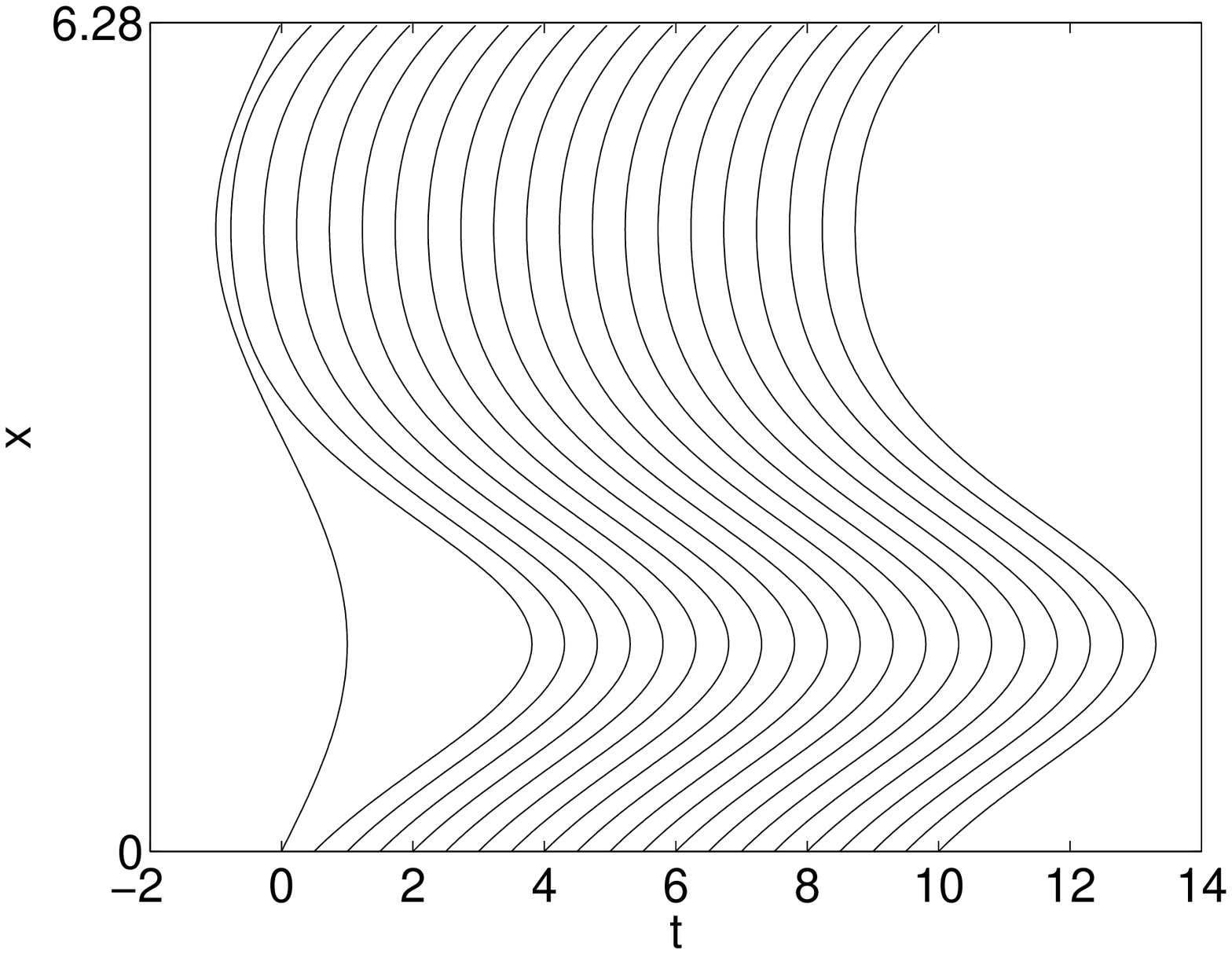}
\end{tabular}
\caption{Front propagation with $\beta=10$, $\varepsilon=0.01$ and $\psi_0(x)=\sin(x)$.}
\label{fig-eps2}
\end{figure}

\begin{figure}[thbp]
 \centering
   \begin{tabular}{c}
\includegraphics[width=9cm,height=3cm,angle=0]{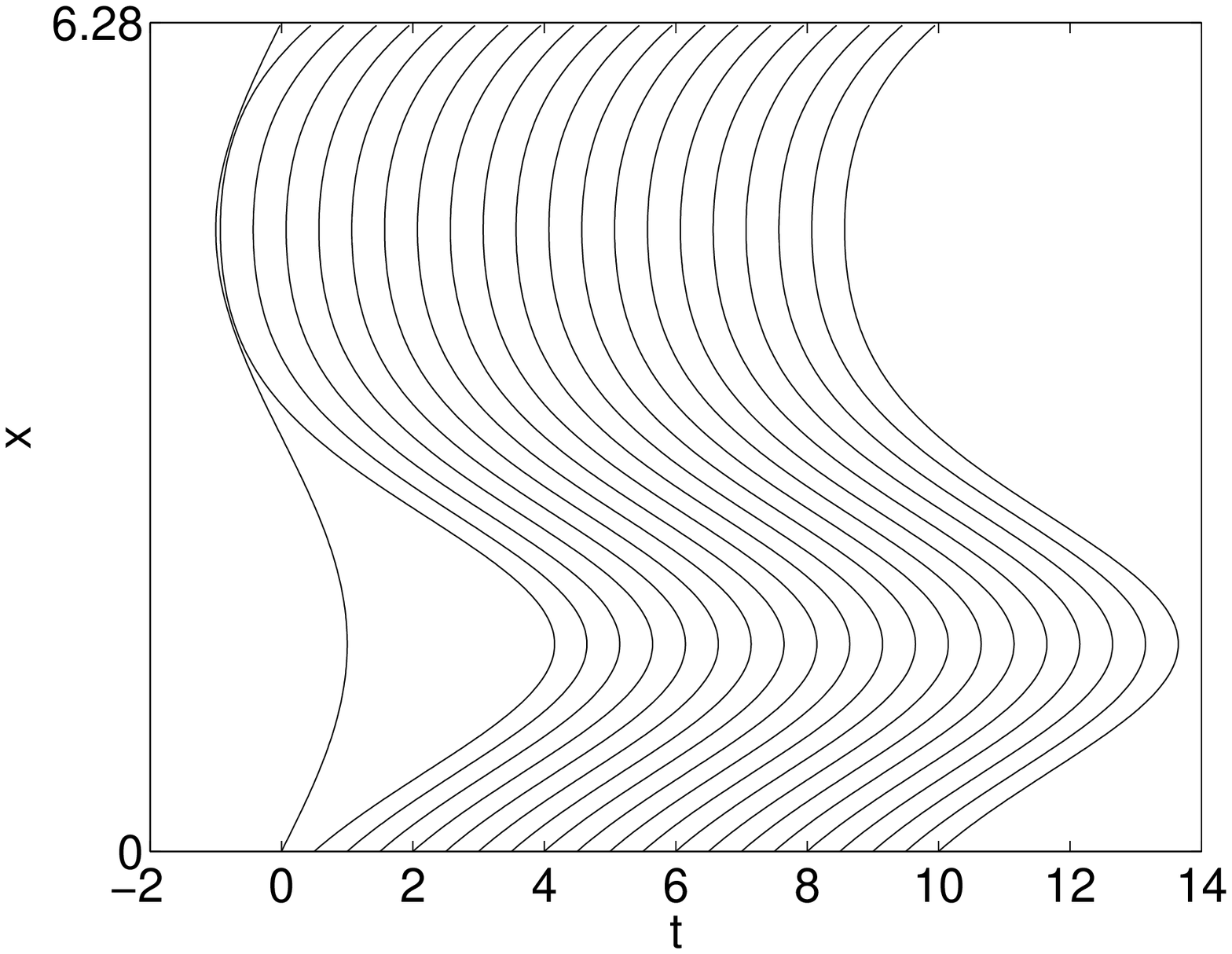}
\end{tabular}
\caption{Front propagation with $\beta=10$, $\varepsilon=0.001$ and $\psi_0(x)=\sin(x)$.}
\label{fig-eps3}
\end{figure}

\begin{figure}[thbp]
 \centering
   \begin{tabular}{c}
\includegraphics[width=9cm,height=3cm,angle=0]{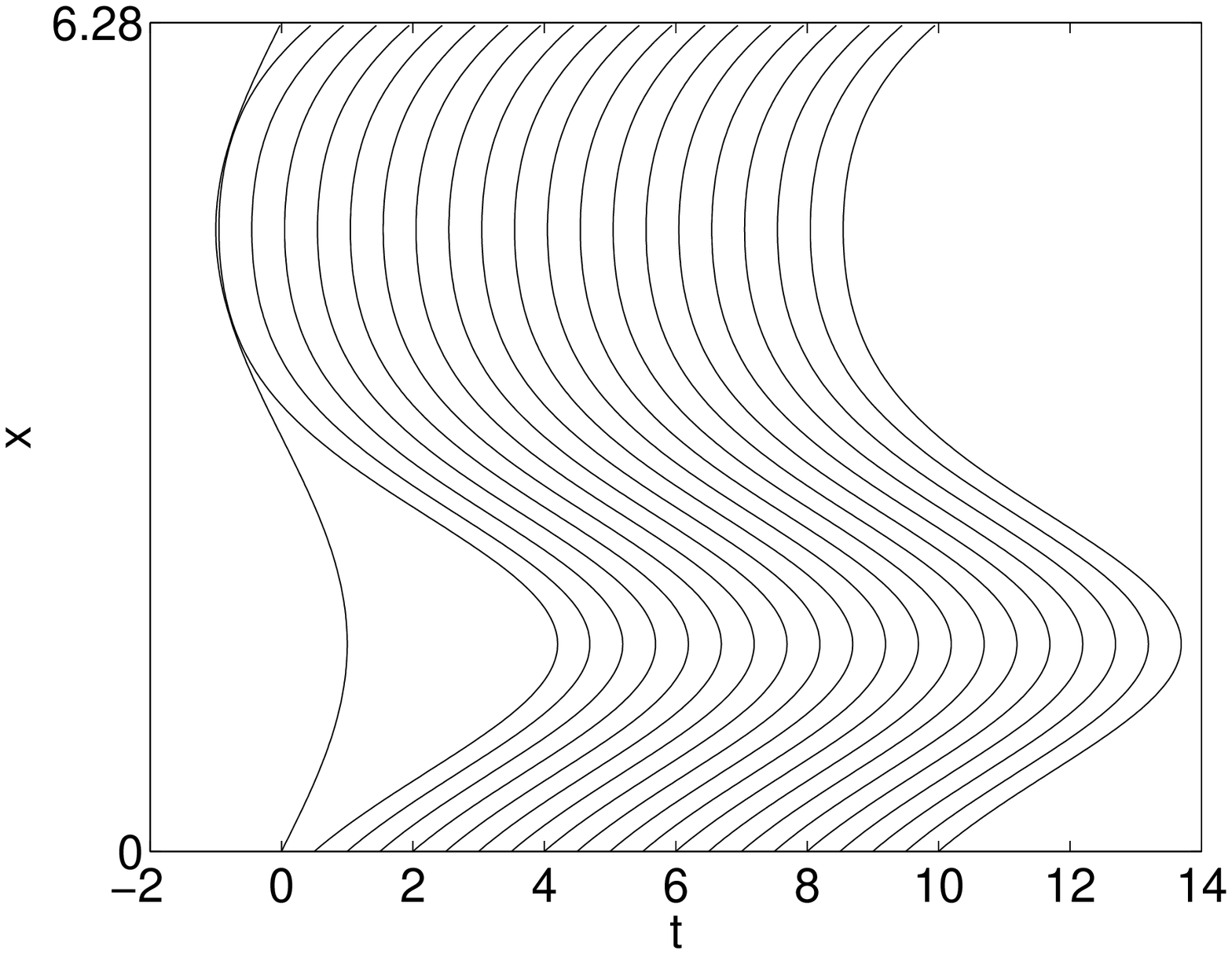}
\end{tabular}
\caption{Front propagation with $ \beta=10$, $\varepsilon=0$ and $\psi_0(x)=\sin(x)$.}
\label{fig-eps4}
\end{figure}

\begin{figure}[thbp]
 \centering
\begin{tabular}{c}
\includegraphics[width=9cm,height=3cm,angle=0]{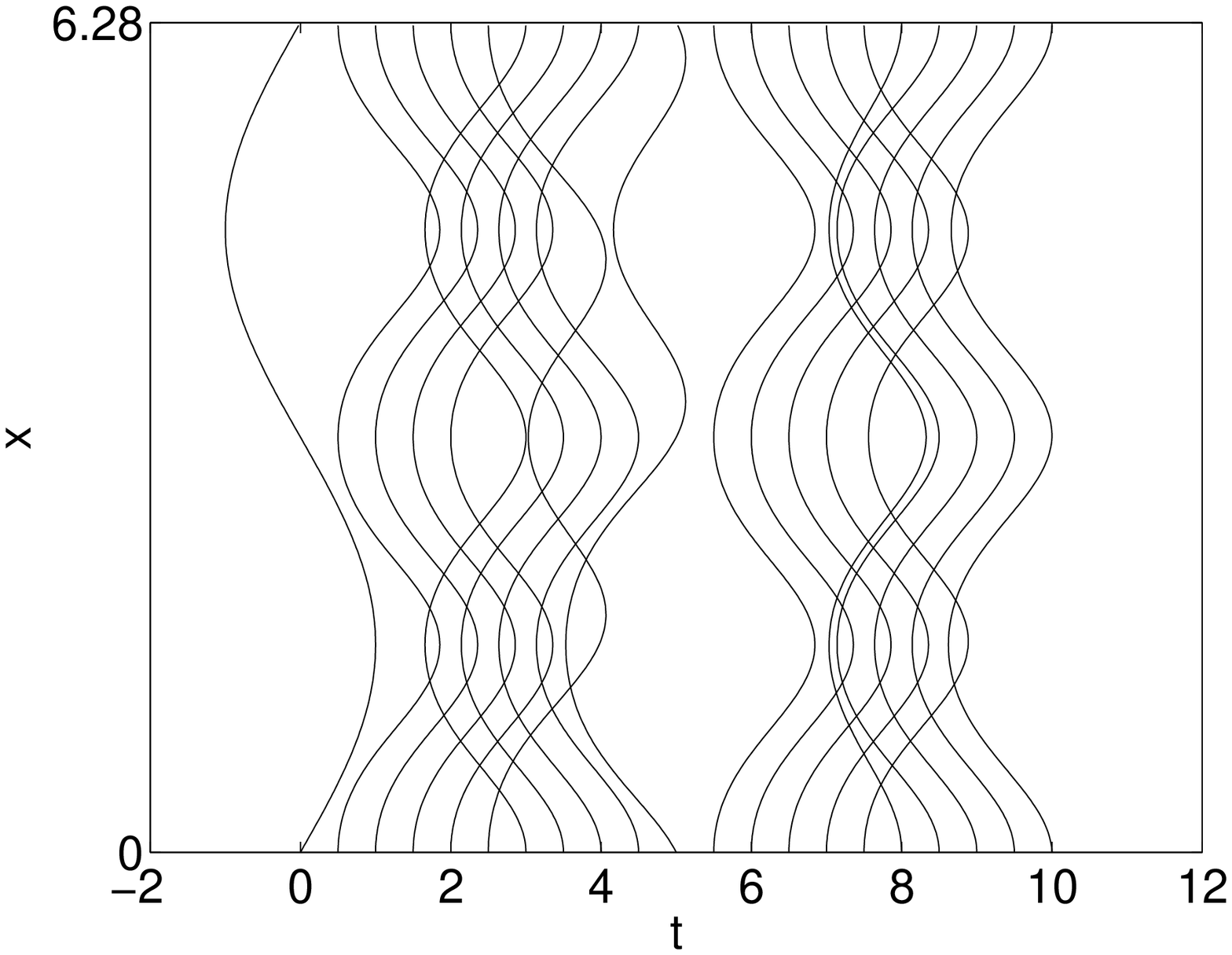}
\end{tabular}
\caption{ Front propagation with $\beta=20$, $\varepsilon=0.1$ and $\psi_0(x)=\sin(x)$.}
\label{fig-eps5}
\end{figure}

\begin{figure}[thbp]
 \centering
   \begin{tabular}{c}
\includegraphics[width=9cm,height=3cm,angle=0]{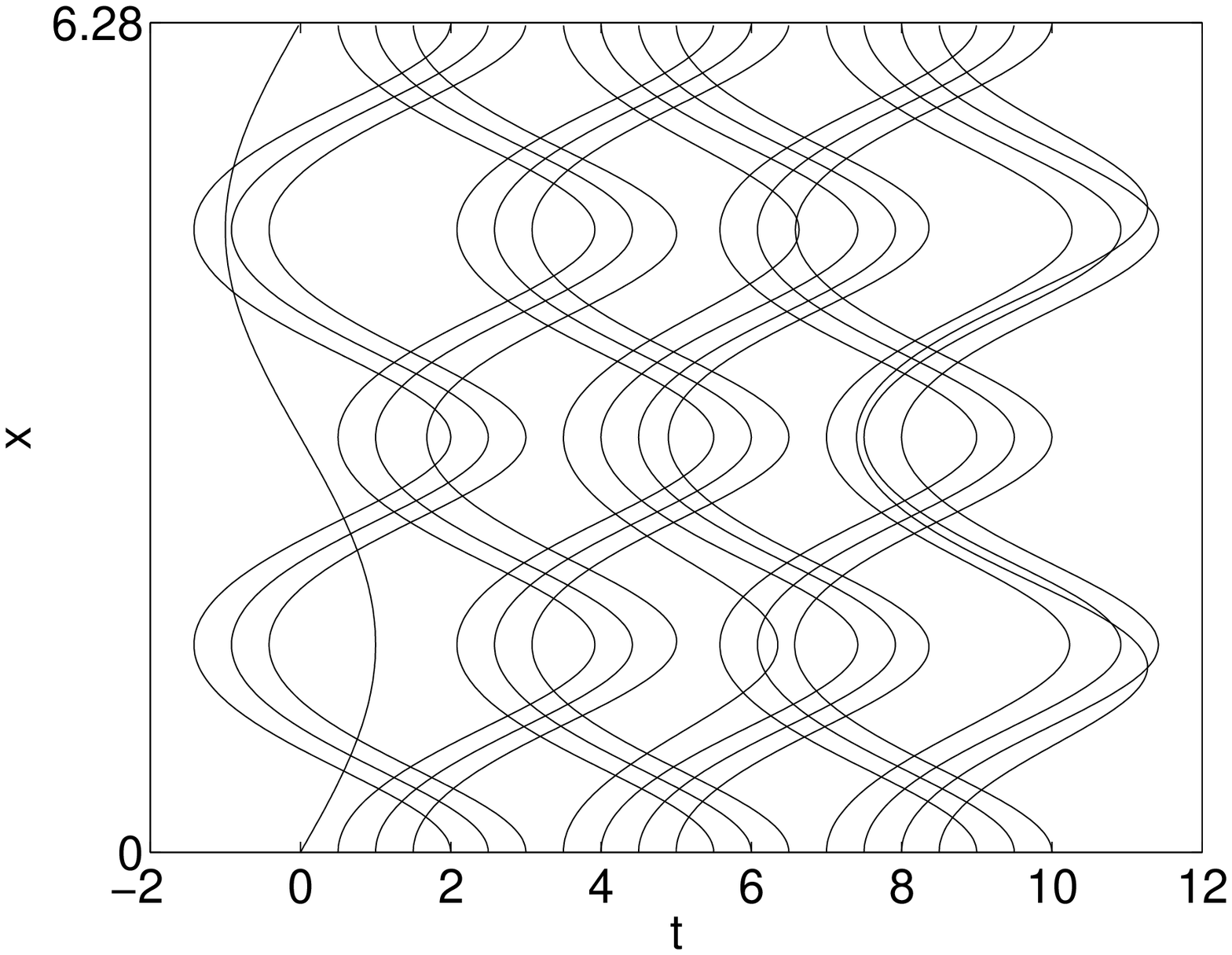}
\end{tabular}
\caption{ Front propagation with $\beta=20$, $\varepsilon=0.01$ and $\psi_0(x)=\sin(x)$.}
\label{fig-eps6}
\end{figure}

\newpage

\begin{figure}[thbp]
 \centering
   \begin{tabular}{c}
\includegraphics[width=9cm,height=3cm,angle=0]{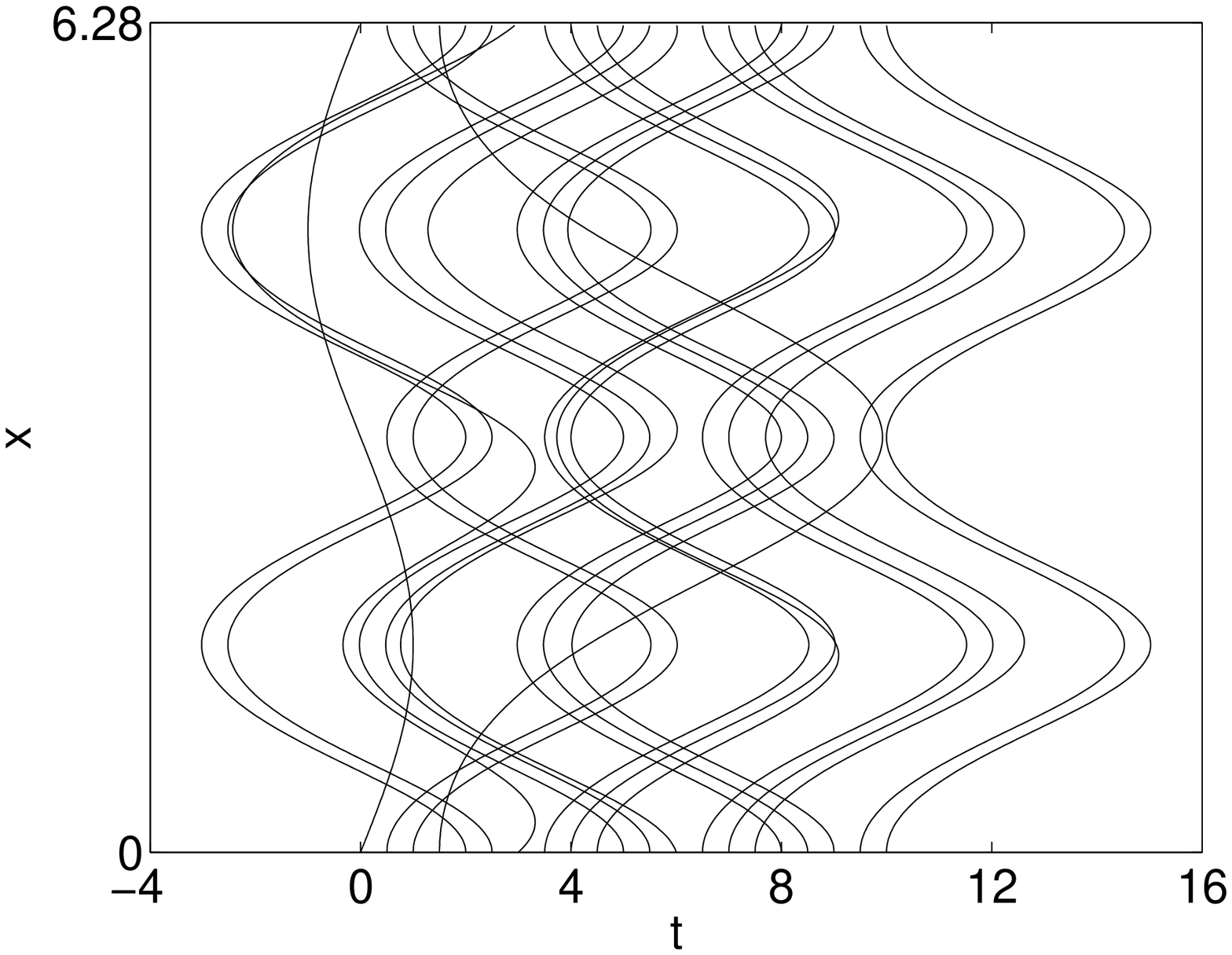}
\end{tabular}
\caption{ Front propagation with $\beta=20$, $\varepsilon=0.001$ and $\psi_0(x)=\sin(x)$.}
\label{fig-eps7}
\end{figure}

\begin{figure}[thbp]
 \centering
   \begin{tabular}{c}
\includegraphics[width=9cm,height=3.2cm,angle=0]{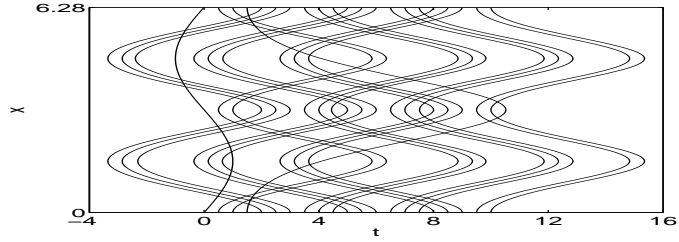}\\[-5mm]
\end{tabular}
\caption{Front propagation with $\beta=20$, $\varepsilon=0$ and $\psi_0(x)=\sin(x)$.}
\label{fig-eps8}
\end{figure}

\newpage

We now investigate the dynamics of Equation \eqref{numerical 1} with respect to the parameter $\beta$. For this purpose,
we fix $\varepsilon=0.001$.

The numerical simulations confirm that, as for the K-S equation, $0$ turns out to be a global attractor
for the solution to Equation \eqref{numerical 1}, for any $\beta\in [1,4]$.
A non-trivial attractor is expected
for larger $\beta$'s.
In Figures \ref{fig-beta1a}-\ref{fig-beta2b}, we can see the front evolutions generated by \eqref{numerical 6} with $\beta=30,\,60$ for two different
initial conditions.
In all the below figures, the periodic orbit is clearly observed.
\begin{figure}[thbp]
 \centering
   \begin{tabular}{c}
\includegraphics[width=9cm,height=3cm,angle=0]{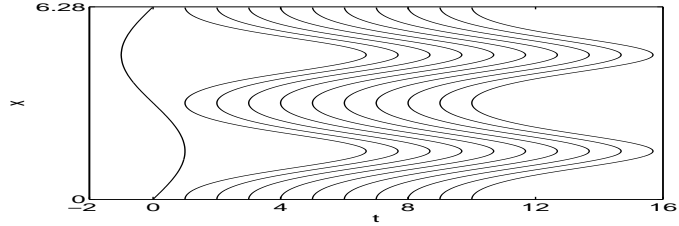}
\end{tabular}
\caption{Front propagation with $\beta=30$, $\varepsilon=0.001$ and $\psi_0(x)=\sin(x)$.}
\label{fig-beta1a}
\end{figure}

\begin{figure}[thbp]
 \centering
   \begin{tabular}{c}
\includegraphics[width=9cm,height=3cm,angle=0]{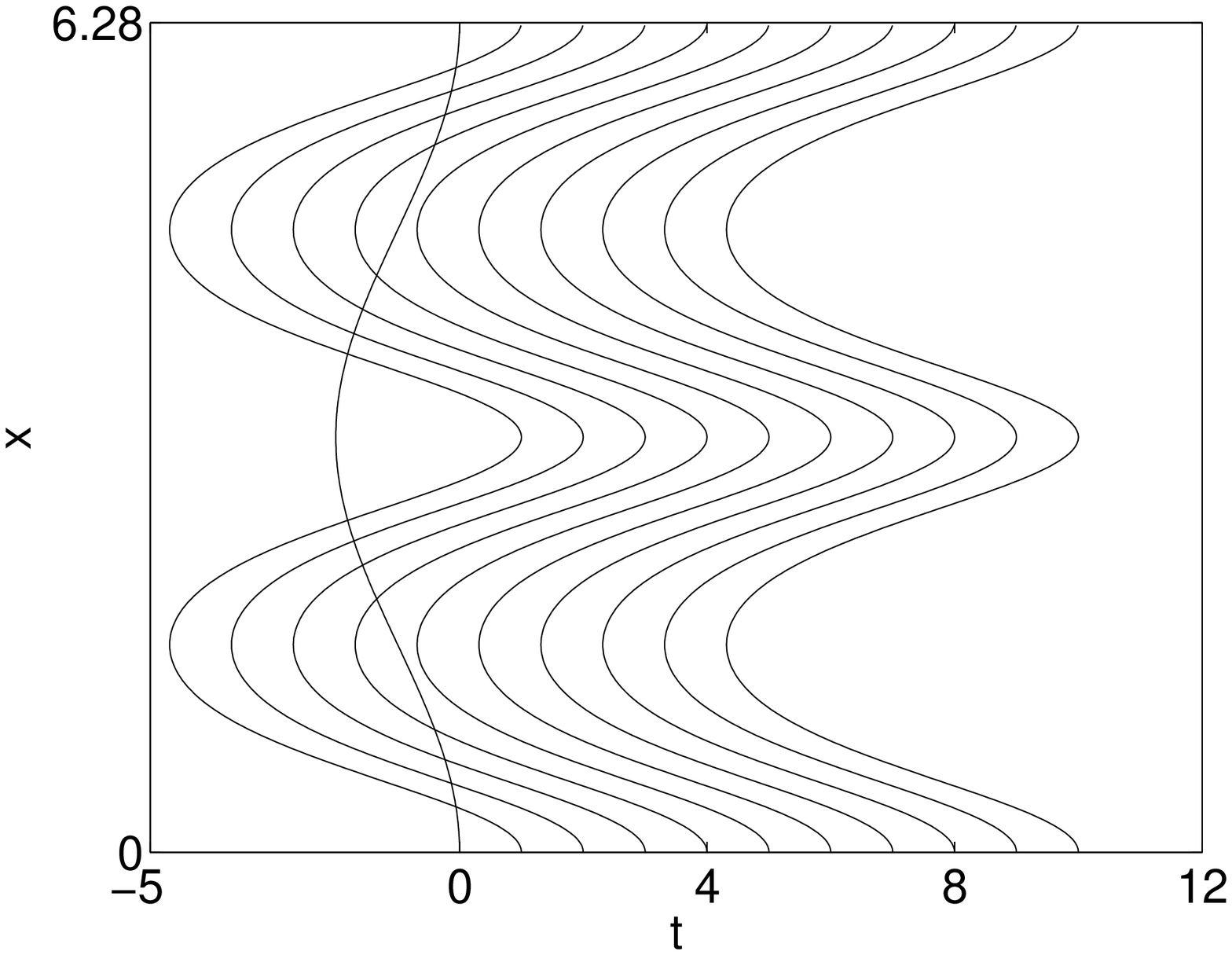}
\end{tabular}
\caption{Front propagation with $ \beta=30$, $\varepsilon=0.001$ and $\psi_0(x)=\cos(x)$.}
\label{fig-beta1b}
\end{figure}

\begin{figure}[thbp]
 \centering
   \begin{tabular}{c}
\includegraphics[width=9cm,height=3cm,angle=0]{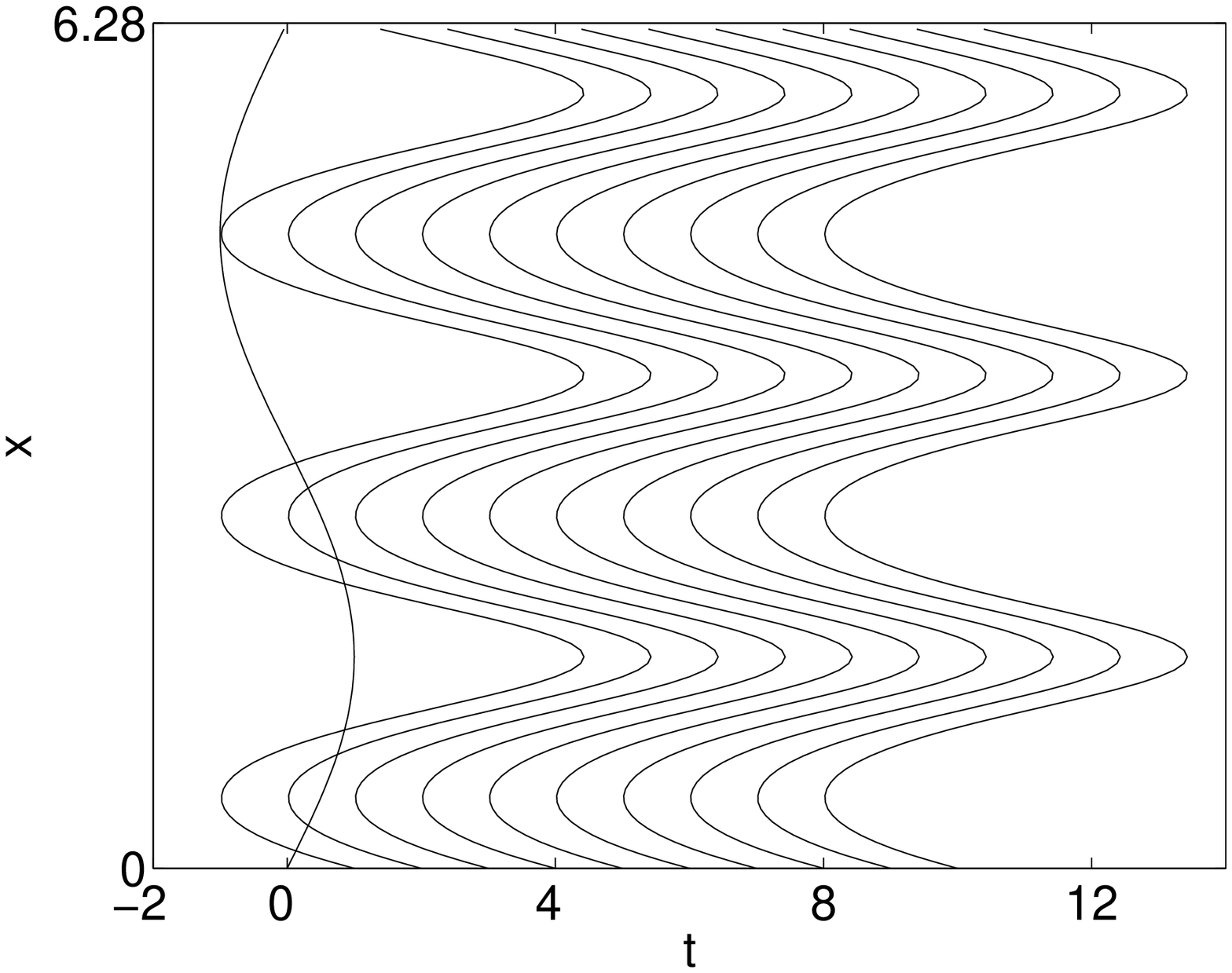}
\end{tabular}
\caption{Front propagation with $ \beta=60$, $\varepsilon=0.001$ and $\psi_0(x)=\sin(x)$.}
\label{fig-beta2a}
\end{figure}

\begin{figure}[thbp]
 \centering
   \begin{tabular}{c}
\includegraphics[width=9cm,height=3cm,angle=0]{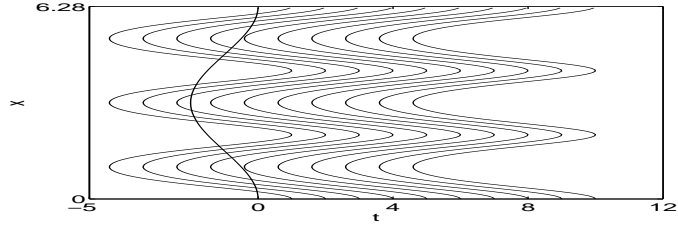}
\end{tabular}
\caption{Front propagation with $ \beta=60, \varepsilon=0.001$, and $\psi_0(x)=\cos(x)$.}
\label{fig-beta2b}
\end{figure}

\newpage

Summing up, our numerical tests confirm that  Equation \eqref{numerical 1} preserves the same structure as K-S equation. Larger $\beta$ generates an even richer dynamics, see  Figure \ref{fig-beta3} where the front propagation is captured from a computation with $\beta=108$. As predicted in \cite{HN}, the front evolves toward an essentially quadrimodal global attractor.
\begin{figure}[thbp]
 \centering
   \begin{tabular}{c}
\includegraphics[width=9cm,height=3cm,angle=0]{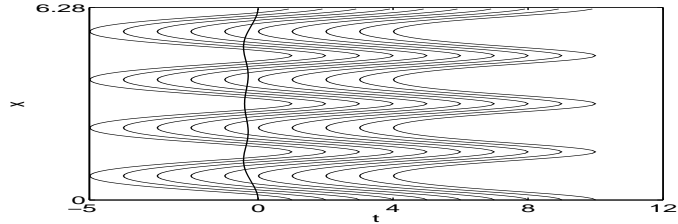}
\end{tabular}
\caption{Front propagation with $ \beta=108, \varepsilon=0.0001$, and $\psi_0(x)=0.1(\cos(x)+\cos(2x)+\cos(3x))$.}
\label{fig-beta3}
\end{figure}

\end{document}